\documentclass[a4paper,11pt]{article}
\usepackage{latexsym,amsfonts,amsmath}
\usepackage{amssymb,amsthm}
\usepackage{color,soul} 
\usepackage{xcolor}
\usepackage[most]{tcolorbox}  

\usepackage[mathscr]{eucal}
\usepackage{enumerate}
\usepackage{enumitem}
\usepackage{xparse}    
\usepackage{etoolbox}  
\usepackage[normalem]{ulem} 

\usepackage{hyperref} 

\usepackage{todonotes} 

\usepackage{pgf}
\usepackage{tikz}
\usetikzlibrary{arrows,automata}

\usepackage{fullpage}

\definecolor{Dgreen}{rgb}{102,255,102}
\definecolor{Purp}{rgb}{102,0,204}

\def\tto{\twoheadrightarrow}

\def\1{\mathbf{1}}
\def\0{\mathbf{0}}
\def\inter{\mathop{\cap}}
\def\NN{\mathbb{N}}

\def\RR{\mathbb{R}}

\def\XX{\mathbf{X}}
\def\AA{\mathbf{A}}
\def\KK{\mathbf{K}}
\def\SS{\mathbf{S}}
\def\HH{\mathbf{H}}
\def\MM{\boldsymbol{\mathcal{M}}}

\newcommand{\widebar}[1]{\overline{#1}}

\newcommand{\bcal}[1]{\boldsymbol{\mathcal{#1}}}
\newcommand{\bscr}[1]{\boldsymbol{\mathscr{#1}}}
\newcommand{\bfrak}[1]{\boldsymbol{\mathfrak{#1}}}

\def\union{\mathop{\cup}}
\def\inter{\mathop{\cap}}

\def\ds{\displaystyle}
\def\nl{\mbox{} \newline }

\DeclareMathOperator\ext{ext}

\definecolor{Dgreen}{rgb}{0,0.5,0}

\usepackage{natbib}

\newtheorem{theorem}{Theorem}[section]
\newtheorem{proposition}[theorem]{Proposition}
\newtheorem{lemma}[theorem]{Lemma}

\newtheorem{definition}[theorem]{Definition}
\newtheorem{remark}[theorem]{Remark}

\theoremstyle{definition}

\renewcommand{\theassumption}{\Alph{assumption}} 
\newlist{assumptions}{enumerate}{1}
\setlist[assumptions,1]{label=(\theassumption.\arabic*), ref=(\theassumption.\arabic*), leftmargin=*}

\NewDocumentEnvironment{List-perso}{m}
  {
    \begin{enumerate}[label=(#1.\arabic*), ref=(#1.\arabic*), leftmargin=*]
  }
  {
    \end{enumerate}
  }

\NewDocumentEnvironment{myenv}{mm}
  {
    \par\medskip
    \noindent\textbf{#1}\par 

    \newlist{vlist}{enumerate}{1}
    \setlist[vlist,1]{label=(#2.\arabic*), ref=(#2.\arabic*), leftmargin=*}
  }
  {
    \par\medskip
  }
  
\newtcolorbox{textecouleur}[1][]{%
  colback=white,
  colframe=white,
  coltext=#1,
  width=\textwidth,
  left=0mm,
  right=0mm,
  box align=left,     
  boxrule=0pt,        
  breakable,
  enhanced,
  sharp corners,
  before skip=5pt,
  after skip=5pt
}
\begin{document}
\bibliographystyle{plain}

\title{On the Feinberg-Piunovskiy Theorem and its extension to chattering policies
\footnote{Supported by grant PID2021-122442NB-I00 from the Spanish \textit{Ministerio de Ciencia e Innovaci\'on.}}}

\author{
Fran\c{c}ois Dufour\footnote{
Bordeaux INP; Inria centre at the University of Bordeaux, Team: ASTRAL; IMB, Institut de Math\'ematiques de Bordeaux, Universit\'e de Bordeaux, France. e-mail:  \tt{francois.dufour@math.u-bordeaux.fr}  {(Author for correspondence)}}
\and
Tom\'as Prieto-Rumeau\footnote{Statistics Department, UNED, Madrid, Spain. e-mail: {\tt{tprieto@ccia.uned.es}}}}
\date{ \today }
\maketitle
\begin{abstract}
The Feinberg-Piunovskiy Theorem established in \cite[Theorem 3.8]{piunovskiy19} asserts that for a discrete-time uniformly absorbing and atomless Markov Decision Process (MDP) with Borel state space and multiple criteria, the family of deterministic stationary policies is a sufficient class of policies. In this paper, we study some related problems and some extensions. In particular dropping the atomless hypothesis,
we establish that the set of chattering stationary policies is a sufficient class of policies for uniformly absorbing MDPs with measurable state space and multiple criteria.
We also prove the Feinberg-Piunovskiy Theorem in the context of a measurable state space in two different ways that differ from \cite{piunovskiy19}.
In particular, we show that the sufficiency of chattering stationary policies directly yields the sufficiency of deterministics stationary policies
for atomless models.
Our approach is partially based on the analysis of extreme points of certain convex sets of occupation measures satisfying integral type constraints.
We show that for a uniformly absorbing model an extreme point of such sets is necessarily given by occupation measures induced by chattering stationary policies of order $d+1$ where $d$ is the dimension of the vector of constraints. 
When in addition the model $\mathsf{M}$ is atomless, then the extreme points of this constrained set of occupation measures are precisely the occupation measures generated by deterministic stationary policies satisfying these constraints.

\end{abstract}
{\small 
}

\section{Introduction}
\label{sec-1}
 
In this paper, we consider a discrete-time absorbing Markov Decision Process (MDP) denoted by $\mathsf{M}$ and defined on a measurable state space $\XX$ and a Borel action space $\AA$, with a fixed initial distribution. 
We adopt a general measurable framework for the state space in order to encompass models commonly considered in game theory.
In particular, non-zero sum games have been recently studied in the context of constraints by using the notion of occupation measure, a concept we study here in details.
The set of pairs of state and  feasible action denoted by $\KK$ satisfies some technical measurability properties and throughout the paper, the hypotheses on $\mathsf{M}$ will be strengthened
by also considering (uniformly) absorbing as well as atomless models.
The set of all policies will be denoted by $\mathbf{\Pi}$ and we will also consider two specific families of policies: the sets of deterministic stationary policies $\mathbb{D}$ and chattering stationary policies $\mathbb{C}_{n}$
of order $n\in\NN^*$,
that is stationary policies given by a randomization (or, in other words, a mixture) between at most $n$ deterministic stationary policies (see Definition \ref{Def-policies}
for a mathematical description of these policies).
We have chosen to use the term of chattering policy which is borrowed from optimal control theory. Readers are referred to the discussion at the top of page 25 in \cite{balder95} and the references therein for a discussion of this topic, as well as 
\cite{feinberg2020} that used it in the context of MDPs.

For a measurable subset $\Delta \subset \XX$ referred to as the absorbing set, a $\Delta$-absorbing MDP is characterized by the fact that, once the state process
enters~$\Delta$, it remains there indefinitely generating no further reward or cost.
A standard assumption is that the expected entrance time of the state process to $\Delta$ is finite under any admissible policy.
The notion of a uniformly absorbing MDP was introduced in \cite[Definition 3.6]{piunovskiy19}) and means roughly speaking that the series of the queues of the hitting time of $\Delta$
go to zero uniformly over the set of policies.
For comprehensive discussions as well as recent results on (uniformly) absorbing MDPs, we refer the reader to~\cite{altman99,fra-tom2024,fra-tom2025,piunovskiy19,feinberg12,Piunovskiy25-book,piunovskiy24,piunovskiy24b,yi24}
and references therein.
An MDP is called atomless if the initial distribution of the state process and the transition probabilities are atomless.

The notion of occupation measure will play a central role in this paper. An occupation measure generated (or induced) by a policy $\pi\in\mathbf{\Pi}$ denoted by $\mu_{\pi}$ provides a concise representation
of the expected aggregated behavior of the state and action processes over time.
More precisely, $\mu_{\pi}(\Gamma)$ gives the expected time spent by the state and action processes in the measurable subset $\Gamma$ of $\XX\times\AA$ under the policy $\pi\in\mathbf{\Pi}$.
We will write $\bcal{O}$ for the set of occupation measures.

A bounded one-step reward vector-function denoted by $r$ with values in $\RR^{d}$ for $d\in\NN^*$ will be used to evaluate the performance of a policy of the model $\mathsf{M}$. More specifically, the performance vector of a policy $\pi\in\mathbf{\Pi}$ will be given by $\ds \mathcal{R}(\pi)=\int_{\XX\times\AA} r(x,a) \mu_{\pi}(dx,da).$
A subset $\Lambda\subset\mathbf{\Pi}$ will be called a sufficient class (family) of policies if $\mathcal{R}(\Lambda)=\mathcal{R}(\mathbf{\Pi})$.

\bigskip

A main, but not the only, purpose of this paper is to characterize the extremes points of sets of occupation measures satisfying 
integral type constraints of the form $\ds \int_{\XX\times\AA} r(x,a) \mu(dx,da)=\alpha$ for some $\alpha\in\RR^d$.
Only a few results have been established on this subject, mainly dealing with the characterization of the extreme points of the unconstrained set of occupation mesures.
Let us mention \cite{altman99,borkar88,borkar91} for finite or countable state spaces, and \cite{feinberg12,piunovskiy24} for Borel state spaces.
If we focus on the literature for Borel MDPs, it has been shown that every extreme (finite) occupation measure is generated by a deterministic stationary strategy
for absorbing models in \cite[Lemma 4.6]{feinberg12} and for models having a costless (rewardless) absorbing state in \cite[Theorem 1]{piunovskiy24}
(in this latter case, this type of model is more general than an absorbing model, in particular an occupation measure may be infinite).
As far as we know, all these results focus on unconstrained sets of occupation measures.
In the present work, we will establish the following results.
\begin{enumerate}[label=\alph*), leftmargin=*, labelsep=0.1cm, itemindent=0pt, topsep=0cm, nosep]
\item \label{convex-a}
We will first show that extreme points of the set of occupation measures satisfying these integral type constraints are necessarily given by occupation measures induced by chattering stationary policies of order $d+1$ satisfying these constraints (where $d$ is the dimension of the reward vector-function).
To get this result, we will show in Proposition \ref{Linearly-bounded-closed} that the set of occupation measures $\bcal{O}$ is linearly closed and linearly bounded.
\item \label{convex-b} If in addition, the model $\mathsf{M}$ is atomless, then the extreme points of this constrained set of occupation measures are precisely
the occupation measures generated by deterministic stationary policies satisfying these constraints.
\end{enumerate}

\bigskip

We refer the reader to Theorems \ref{Dubins} and \ref{Castaing-Valadier} respectively for a rigorous and complete statement of these two results.
As far as we know, they are both novel and will play a particularly important role in showing the sufficiency of certain classes of policies as explained now.

\bigskip

The second main objective of this paper is to study for a uniformly absorbing model, the sufficiency of the set of chattering stationary policies and also the sufficiency of the deterministic stationary policies but in assuming in addition that the model is atomless.
There are few results on the sufficiency of chattering policies for MDPs even in the case of a real valued reward function.
To the best of our knowledge, the only results of this type are related to optimization problems, see for example Theorem 2 in \cite{feinberg2020}, Theorem 9.2 in \cite{feinberg12} 
and Theorem 2 in \cite{piunovskiy24} where roughly speaking, it is shown that under continuity-compactness conditions, the optimal policy for a constrained absorbing MDP is given
by a chattering policy of order $n+1$ where $n$ is the number of constraints. 
It is also worth mentioning the comment at the top of page 167 in \cite{piunovskiy19} stating that for a discounted MDP
with finite state and action spaces, $\mathcal{R}(\mathbf{\Pi})$ is a convex hull of $\mathcal{R}(\mathbb{D})$ by virtue of Theorem 6.1 in \cite{feinberg12}.
It is very interesting to see that another way of stating this result is to say that the family of chattering policies is a sufficient class of policies which is precisely one of the main results
we will establish in this paper in a very general framework.
There are more references on the sufficiency of deterministic policies for atomless and uniformly absorbing models
and we refer the reader to the introduction of \cite{piunovskiy19} for a comprehensive and detailed panorama on this topic for total expected reward as well as for discounted reward..
We will focus on providing a brief description of the papers that we believe are most relevant to our results.
In \cite{piunovskiy00ef,piunovskiy02ef} the authors studied the sufficiency of deterministic Markov policies for atomless models with Borel state space in a very general context by considering a vector performance functional having the form of total rewards (\textit{i.e.} the model under consideration being not necessarily absorbing).
Regarding the sufficiency of deterministic stationary policies for absorbing models, one of the most striking result is certainly the Feinberg-Piunovskiy Theorem, see Theorem 3.8 in \cite{piunovskiy19}.
We should also mention \cite{feinberg-piunovskiy06} where the authors  have shown the sufficiency of nonrandomized decision rules for statistical decision problems with nonatomic state distributions 
holds for arbitrary Borel decision sets and for arbitrary measurable loss functions generalizing the celebrated Dvoretzky-Wald-Wolfowitz Theorem.
In the same vein, it is also worth mentioning \cite{jaskiewicz-nowak19} where the authors provide a generalization of this Theorem to the case of conditional expectations.

\bigskip

\noindent 
In the present paper, we will establish the following results:
\begin{enumerate}[label=\arabic*), leftmargin=*, labelsep=0.1cm, itemindent=0pt, topsep=0cm]
\item \label{claim-suff1} Assuming the model $\mathsf{M}$ is uniformly absorbing, for any policy $\pi\in\mathbf{\Pi}$ there exists a chattering stationary policy
$\gamma\in\mathbb{C}_{d+1}$ of order $d+1$ having the same performance vector, that is,
$\ds \int_{\XX\times\AA} r(x,a) \mu_{\pi}(dx,da) = \int_{\XX\times\AA} r(x,a) \mu_{\gamma}(dx,da)$
or equivalently, $\mathcal{R}(\mathbb{C}_{d+1})=\mathcal{R}(\mathbf{\Pi})$, see Theorem \ref{Main-theorem-aa}.
As far as we know, this result is to the best of our knowledge novel even in the case of a real valued reward function.
To get it, we will first show that the set of chattering Markov policies of order $2d+1$ is a sufficient class of control strategies
(Theorem \ref{deterministic-Markov-policy}\ref{deterministic-Markov-policy-a}).
Combined with the property stated in point \ref{convex-a}, this yields the following important properties:  the set $\mathcal{R}(\mathbb{C}_{d+1})$
is convex and dense in $\mathcal{R}(\mathbf{\Pi})$ and have also the same relative interior as $\mathcal{R}(\mathbf{\Pi})$ (Theorem \ref{Convexity-Denseness}\ref{Convexity-Denseness-a}).
The claim is then obtained by showing a technical property of uniformly absorbing models. Roughly speaking, for any policy $\pi\in\mathbf{\Pi}$, there exists a subset 
$\mathbf{\Pi}^{*}$ of $\mathbf{\Pi}$ such that the model associated with $\mathbf{\Pi}^{*}$  inherits the properties of $\mathsf{M}$ 
and for which  $\mu_{\pi}(r)$ is in the relative interior of $r(\mathbf{\Pi}^{*})$ (see Theorem \ref{Interior-point} for a precise statement of this result).
This result partly builds on a generalization of a theorem established in \cite[Theorem 6.3]{piunovskiy19} under the assumption that the model is atomless, extending it to the setting where this assumption is not imposed and for 
a general measurable state space.
All the properties of the set $\mathcal{R}(\mathbb{C}_{d+1})$ we have outlined above appear to be new.
\item \label{claim-suff2} If in addition the model is atomless, then the set of deterministic stationary policies is sufficient: 
$\mathcal{R}(\mathbb{D})=\mathcal{R}(\mathbf{\Pi})$, see Theorem \ref{Main-theorem-bb}.
This is precisely the Feinberg-Piunovskiy Theorem in the context of a general state space.
We will propose two different proofs for that result. 
One proof is based on Proposition \ref{Consequence-chattering} that states that 
the sufficiency of chattering stationary policies directly yields the sufficiency of deterministic stationary policies
for atomless models by using the characterization of extreme points of sets of occupation measures satisfying integral type constraints.
The second proof of this claim proceeds along the same lines as for item \ref{claim-suff1}.
We now highlight the main differences of this second approach with the one given in \cite[Theorem 3.8]{piunovskiy19}.
The same result is obtained here but for more general models since the state space is allowed to be a measurable space, which is a notable first difference.
Indeed, the approach used in  \cite{piunovskiy19} is based on  \cite[Theorem 2.1]{piunovskiy02ef}  which depends heavily on the Borel structure of the state space $\mathbf{X}$.
This last point was emphasized in \cite[Remark 5]{jaskiewicz-nowak19}.
To illustrate another important difference, observe that as mentioned in \cite{piunovskiy19}: 
\textit{one and nontrivial step is to prove that for an atomless model the set of performance vectors for all deterministic stationary policies is convex.
This fact is nontrivial even for the case of one criterion.}
In this regard, it should be emphasized that the authors in \cite{piunovskiy19} show that the convexity of the set $\mathcal{R}(\mathbb{D})$ is a consequence of their main result
(see Corollary 3.10 and the discussion above in \cite[Page 169]{piunovskiy19}).
The approach used in \cite{piunovskiy19} is based on a dimensionality reduction technique as the authors themselves state.
It consists of showing that the set of occupation measures generated by deterministic stationary policies endowed with the setwise topology is path connected
and then proceed by induction on the dimension of the one-step reward vector-function $r$ to get the sufficiency of the set of deterministic stationary policies, see section 9 in \cite{piunovskiy19}.
The proofs of the convexity of $\mathcal{R}(\mathbb{D})$ and the sufficiency of the set of deterministic stationary policies are intertwined through this recurrence.
As for us, the convexity of $\mathcal{R}(\mathbb{D})$ is directly obtained by using the characterization of the extreme points of the constrained sets of occupation measures.
The sufficiency of the family $\mathbb{D}$ of policies is then obtained by using the convexity of $\mathcal{R}(\mathbb{D})$ and the fact that this set has the same relative interior of $\mathcal{R}(\mathbf{\Pi})$
highlighting the differences between our two approaches.
Another important difference is that the notion of \textit{relative interior} of a convex set plays a central role in our work.
\end{enumerate}

\bigskip

The rest of the paper is organized as follows. In the remaining of this section we introduce some notation and two results from Young measures theory that will be useful in the sequel.
In Section \ref{Sec-2}, we introduce the control model under consideration and provide some basic definitions as well as basic properties.
Section \ref{Sec-3} is devoted to the characterization of the extreme points of the some constrained sets of occupation measures.
Under a set of appropriate assumption, we will establish in section \ref{Sec-4} a preliminary result on the sufficiency of the families of chattering and deterministic Markov policies.
We will reinforce this result by showing in the last section that the families of chattering and deterministic stationary policies are sufficient.

\paragraph{Notation and terminology.}
$\NN$ is the set of natural numbers including $0$, $\NN^{*}=\NN-\{0\}$, $\RR$ denotes the set of real numbers, $\RR_{+}$ the set of non-negative real numbers,
$\RR_{+}^{*}=\RR_{+}-\{0\}$, $\widebar{\RR}_{+}=\RR_{+}\union \{+\infty\}$ and $\widebar{\RR}_{+}^*=\RR_{+}^*\union \{+\infty\}$.
For any $q\in \NN$, $\NN_{q}$ is the set $\{0,1,\ldots,q\}$ and for any $q\in \NN^{*}$, $\NN_{q}^{*}$ is the set $\{1,\ldots,q\}$.
We write $\bscr{S}_{q}=\big\{(\beta_{1},\ldots,\beta_{q})\in\RR_{+}^q : \sum_{i=1}^q \beta_{i}=1 \big\}$ for the standard simplex in $\RR^q$.
We denote by $0_p$ the zero vector in $\mathbb{R}^p$ for $p \in \mathbb{N}^*$. If $\{u_{n}\}_{n\in\NN}$ is a sequence of real numbers, we will the following convention 
$\sum_{n=0}^{-1} u_{n}=0$ to simplify the writing of certain formulae.

On a measurable space $(\mathbf{\Omega},\bcal{F})$ we will consider the set of finite signed measures $\bcal{M}(\mathbf{\Omega})$, the set $\bcal{M}^+(\mathbf{\Omega})$
of finite nonnegative measures, 
and the set of probability measures~$\bcal{P}(\mathbf{\Omega})$.
An element of $\bcal{M}^+(\mathbf{\Omega})$ will be simply called a measure.
Let $(\mathbf{\Omega},\bcal{F},\mu)$ be a measure space. A set $\Gamma\subset\mathbf{\Omega}$ is called a $\mu$-null set if there exists $\Lambda\in\bcal{F}$ such that
$\Gamma\subset\Lambda$ and $\mu(\Lambda)=0$.
For a set $\Gamma\in\bcal{F}$, we denote by $\mathbf{I}_{\Gamma}:\mathbf{\Omega}\rightarrow\{0,1\}$ the indicator function of the set~$\Gamma$, that is,
$\mathbf{I}_{\Gamma}(\omega)=1$ if and only if $\omega\in\Gamma$.
For $\omega\in\mathbf{\Omega}$, we write $\delta_{\{\omega\}}$ for the Dirac probability measure at $\omega$ defined on $(\mathbf{\Omega},\bcal{F})$ by
$\delta_{\{\omega\}}(B)=\mathbf{I}_{B}(\omega)$ for any $B\in\bcal{F}$.

Let $(\mathbf{\Omega},\bcal{F})$ and $(\widetilde{\mathbf{\Omega}},\widetilde{\bcal{F}})$ be two measurable spaces.
A kernel on $\widetilde{\mathbf{\Omega}}$ given $\mathbf{\Omega}$ is a mapping
\mbox{$Q:\mathbf{\Omega}\times\widetilde{\bcal{F}}\rightarrow\RR^+$} such that $\omega\mapsto Q(B|\omega)$ is 
measurable on $(\mathbf{\Omega},\bcal{F})$ for every set $B$ in $\widetilde{\bcal{F}}$,   and  $B\mapsto Q(B|\omega)$ is in $\bcal{M}^+(\widetilde{\mathbf{\Omega}})$
for every $\omega\in\mathbf{\Omega}$. Depending on the context, and in particular when it is important to specify the $\sigma$-algebra associated with $\mathbf{\Omega}$, we will say that $Q$ is a kernel 
on $\widetilde{\mathbf{\Omega}}$ given $(\mathbf{\Omega},\bcal{F})$. Moreover, if $Q$ is a kernel on $\mathbf{\Omega}$ given $\mathbf{\Omega}$, we will simply say that
$Q$ is a kernel on $\mathbf{\Omega}$.
If $Q(\widetilde{\mathbf{\Omega}}|\omega)=1$ for all $\omega\in\mathbf{\Omega}$ then we say that $Q$ is a \textit{stochastic} kernel.
We now introduce two types of specific kernels we will work with.
\begin{enumerate}
\item A kernel $\gamma$ on $\widetilde{\mathbf{\Omega}}$ given $\mathbf{\Omega}$ is called \textit{finitely supported of order $p\in\NN^{*}$} if there exist $p$ measurable functions
$\{\phi_{i}\}_{i\in\NN_{p}^{*}}$ defined from $\mathbf{\Omega}$ to $\widetilde{\mathbf{\Omega}}$ and a measurable function $\beta$ defined 
from $\mathbf{\Omega}$ to $\bcal{S}_{p}$ satisfying
$$ \gamma(d\widetilde{\omega} | \omega) =\sum_{i=1}^{p} \beta_{i}(\omega) \delta_{\phi_{i}(\omega)}(d\widetilde{\omega})$$
for any $\omega\in\mathbf{\Omega}$ where $\beta(\omega)=(\beta_{1}(\omega),\ldots,\beta_{p}(\omega))$.
This terminology is borrowed from the theory of Young measures, see Definition 8.1 in \cite{balder95}.
\item For $\Gamma\in\bcal{F}$, we write $\mathbb{I}_{\Gamma}$ for the kernel on $\mathbf{\Omega}$ defined by
$\mathbb{I}_{\Gamma}(B|\omega)=\mathbf{I}_{\Gamma}(\omega) \delta_{\{\omega\}}(B)$ for $\omega\in\mathbf{\Omega}$ and $B\in\bcal{F}$.
\end{enumerate}
Let $Q$ be a stochastic kernel on $\widetilde{\mathbf{\Omega}}$ given $\mathbf{\Omega}$.
For  a bounded measurable function $f:\widetilde{\mathbf{\Omega}}\rightarrow\RR^p$ or a measurable function $f:\widetilde{\mathbf{\Omega}}\rightarrow\widebar{\RR}^p_{+}$
with $p\in\NN^*$, we will denote by $Qf$ the measurable function defined on $\mathbf{\Omega}$ by
$Qf=(Qf_{1},\ldots,Qf_{i},\ldots,Qf_{p})$ where $\ds Qf_{i}(\omega)=\int_\mathbf{\widetilde{\Omega}} f_{i}(\widetilde{\omega})Q(d\widetilde{\omega}|\omega)$ with $f_{i}$ is the $i$-th component of $f$.
for $\omega\in\mathbf{\Omega}$, $i\in\{1,\cdots,p\}$.
In the same spirit, for a measure $\mu\in\bcal{M}^{+}(\mathbf{\Omega})$, we write $\mu(f)$ for the vector $\big(\mu(f_{1}),\ldots,\mu(f_{i}),\ldots,\mu(f_{p})\big)$
where $\ds \mu(f_{i})=\int_{\mathbf{\Omega}} f_{i}(\omega) \mu(d\omega)$.
We also denote by $\mu Q$ the finite measure  on $(\widetilde{\mathbf{\Omega}},\widetilde{\bcal{F}})$ given by
$\ds B\mapsto \mu Q\,(B)= \int_{\mathbf{\Omega}} Q(B|\omega) \mu(d\omega)$ for $B\in\widetilde{\bcal{F}}$.
Let $(\widebar{\mathbf{\Omega}},\widebar{\bcal{F}})$ be a third measurable space and $R$ be a stochastic kernel on $\widebar{\mathbf{\Omega}}$
given $\widetilde{\mathbf{\Omega}}$. Then we will denote by $QR$ the stochastic kernel on $\widebar{\mathbf{\Omega}}$ given $\mathbf{\Omega}$ defined as
$\ds QR(\Gamma |\omega)= \int_{\widetilde{\mathbf{\Omega}}} R(\Gamma | \tilde{\omega}) Q(d\tilde{\omega} | \omega)$
for $\Gamma\in\widebar{\bcal{F}}$ and $\omega\in\bcal{F}$.
The product of the $\sigma$-algebras $\bcal{F}$ and $\widetilde{\bcal{F}}$ is denoted by $\bcal{F}\otimes\widetilde{\bcal{F}}$ and consists of the $\sigma$-algebra
generated by the measurable rectangles, that is, the sets of the form $\Gamma\times\widetilde{\Gamma}$ for $\Gamma\in\bcal{F}$ and
$\widetilde{\Gamma}\in\widetilde{\bcal{F}}$.
We denote by $\mu\otimes Q$ the unique  finite measure on the product space $(\mathbf{\Omega}\times\widetilde{\mathbf{\Omega}},\bcal{F}\otimes\widetilde{\bcal{F}})$ satisfying 
$\ds \mu\otimes Q( \Gamma\times\widetilde{\Gamma})= \int_{\Gamma} Q(\widetilde{\Gamma}|\omega)\mu(d\omega)$ for $\Gamma\in\bcal{F}$ and
$\widetilde{\Gamma}\in\widetilde{\bcal{F}}$.
Given $\mu\in\bcal{M}(\mathbf{\Omega}\times\widetilde{\mathbf{\Omega}})$, 
$\mu^{\mathbf{\Omega}}(\cdot)=\mu(\cdot\times \widetilde{\mathbf{\Omega}})\in\bcal{M}(\mathbf{\Omega})$ and
$\mu^{\widetilde{\mathbf{\Omega}}}(\cdot)=\mu(\mathbf{\Omega}\times\cdot)\in\bcal{M}(\widetilde{\mathbf{\Omega}})$ are the marginal measures.

For a metric space $\SS$, we write $\bfrak{B}(\SS)$ for its Borel $\sigma$-algebra. A metric space will be always endowed with its Borel $\sigma$-algebra.
A subset of a metric space will always be equipped with the induced metric unless explicitly stated otherwise.
For a measurable space $(\mathbf{\Omega},\bcal{F})$ and a metric space $\SS$, $\bcal{L}^{0}_{\SS}(\mathbf{\Omega},\bcal{F})$ stands for the set of measurable functions from $(\mathbf{\Omega},\bcal{F})$
into $(\SS,\bfrak{B}(\SS))$. If there is no ambiguity about the $\sigma$-algebra associated with $\mathbf{\Omega}$, we will write $\bcal{L}^{0}_{\SS}(\mathbf{\Omega})$ to simplify notation.
If $\mathbf{S}$ is a Polish space (a complete and separable metric space), we will consider on $\bcal{M}(\mathbf{\Omega}\times\mathbf{S})$ the  $ws$-topology (weak-strong topology)
which is the coarsest topology for which the mappings $\mu\mapsto \mu(f)$ are continuous for any bounded Carath\'eodory function $f$ defined on $\mathbf{\Omega}\times\mathbf{S}$.

\bigskip

There are several definitions of extreme set in the literature. We use here Definition 7.61 in \cite{aliprantis06}.
An extreme subset of a subset $\mathcal{C}$ of a vector space $\mathcal{V}$ is a nonempty subset $\mathcal{E}$ of $\mathcal{C}$ with the property that
if $x$ belongs to $\mathcal{E}$ it cannot be written as a proper convex combination of points of $\mathcal{C}$ outside $\mathcal{E}$.
A point $x$ is an extreme point of $\mathcal{C}$ if the singleton $\{x\}$ is an extreme set.
The set of extreme points of a subset $\mathcal{C}$ of a vector space $\mathcal{V}$ is denoted by $\ext(\mathcal{C})$.
A face of $\mathcal{C}$ is an extreme subset of $\mathcal{C}$ that is itself convex.
A convex set $\mathcal{C}$ in a vector space $\mathcal{V}$ is linearly bounded (linearly closed, respectively) if the intersection of $\mathcal{C}$ with any line in $\mathcal{V}$ is a bounded (closed, respectively) set.
The relative (respectively, boundary) interior of a convex set $\mathcal{C}$ in a topological vector space $\mathcal{V}$ is defined to be its topological (respectively, boundary) interior relative to its affine hull.

\bigskip

Let us recall the next disintegration lemma that will be useful in the forthcoming.
\begin{lemma}[Disintegration lemma]
\label{lemma-disintegration} 
Let $(\XX,\bcal{F})$ be a measurable space and $\AA$ a Borel space. Assume that $\KK=\{(x,a)\in\XX\times\AA : a\in\AA(x) \} \in\bcal{F}\otimes\bfrak{B}(\AA)$ and suppose also
that there exists a stochastic kernel $\kappa$ on $\AA$ given $\XX$ satisfying  $\kappa(\AA(x)|x)=1$ for any $x\in\XX$.
For every measure $\mu\in\bcal{M}^{+}(\XX\times\AA)$ such that $\mu(\mathbf{K}^{c})=0$ there exists a stochastic kernel $\sigma$ on $\AA$ given $\XX$ satisfying 
$\sigma(\AA(x)|x)=1$ for any $x\in\XX$ and $\mu= \mu^{\XX}\otimes \sigma$. There is uniqueness in the following sense: if $\sigma$ and $\sigma'$ are two stochastic kernels on $\AA$ given $\XX$ satisfying
$\mu= \mu^{\XX}\otimes \sigma= \mu^{\XX}\otimes \sigma'$ with $\sigma(\AA(x)|x)=\sigma'(\AA(x)|x)=1$ for any $x\in\XX$ 
then, $\sigma(\cdot|x)=\sigma'(\cdot|x)$ for $\mu^\XX$-almost every $x$.
\end{lemma}
\begin{proof}
By using Corollary 10.4.15 in \cite{bogachev07}, we easily get the result.
\end{proof}
The following result is a version of Theorem 8.2 in \cite{balder95} under different assumptions. It can be proved by following the same reasoning as for Theorem 8.2 in \cite{balder95}.
\begin{theorem}[Finitely supported Young-measures]
\label{finitely-supported-young-measure}
Let $(\XX,\bcal{F},\mu)$ be a measure space, a family of subsets $\{\AA(x)\}_{x\in\XX}$ of a Borel space $\AA$. Consider 
$g\in\bcal{L}^{0}_{\RR^{n}_{+}}(\XX\times\AA)$ for $n\in\NN^{*}$ and a stochastic kernel $\kappa$ on $\AA$ given $\XX$ satisfying $\kappa(\AA(x)|x)=1$ for any $x\in\XX$  and
$\ds \int_{\XX\times\AA} |g(x,a)| \kappa(da | x) \mu(dx) < +\infty$.
Suppose also that the set $\KK=\{(x,a)\in\XX\times\AA : a\in\AA(x) \} \in\bcal{F}\otimes\bfrak{B}(\AA)$ and there exists a function $\theta\in\bcal{L}^{0}_{\AA}(\XX)$ whose graph is a subset of $\KK$.
Then, there exists a finitely supported kernel $\gamma$ on $\AA$ given $(\XX,\bcal{F})$ of order $n+1$ such that 
\begin{enumerate}[label=(\alph*)]
\item For any $x$ in $\XX$, $\gamma(\AA(x)|x)=1$.
\item For $\mu$-almost every $x$ in $\XX$, $\ds \int_{\AA} g(x,a) \kappa(da | x)  =  \int_{\AA} g(x,a) \gamma(da | x).$
\end{enumerate}
\end{theorem}

\section{The Markov control model}
\label{Sec-2}
The main goal of this section is to introduce the parameters defining the control model with a brief presentation of the construction of the controlled process.
We will present different types of properties for a control model, such as being (uniformly) absorbing and atomless. 
We will also introduce the notion of occupation measure, which has been studied in detail in \cite{fra-tom2024,fra-tom2025,piunovskiy24} for absorbing models and
which will play a fundamental role in the proof techniques we will use in this paper.

\subsection{The model.}
\label{Description-model}
A control model  $\mathsf{M}$ consists of the following parameters.
\begin{itemize}
\item A state space $\XX$ endowed with a $\sigma$-algebra $\bfrak{X}$. The set $\XX$ will always be endowed with the $\sigma$-algebra $\bfrak{X}$ unless explicitly stated, in which case the context will be specified.
\item A Borel space $\mathbf{A}$, representing the action space.
\item A family of nonempty measurable sets $\AA(x)\subseteq \AA$ for $x\in\XX$. The set $\AA(x)$ gives the admissible actions in state~$x$.
Let $\KK=\{(x,a)\in\XX\times\AA: a\in \AA(x)\}$ be the family of feasible state-action pairs.
\item A stochastic kernel $Q$ on $\mathbf{X}$ given $\big(\XX\times\AA,\bfrak{X}\otimes\bfrak{B}(\AA)\big)$ which stands for the transition probability function. 
\item An initial distribution given by $\eta\in\bcal{P}(\XX)$.
\item A bounded reward vector-function $r\in\bcal{L}^{0}_{\RR^{d}}(\XX\times\AA)$ for $d\in\NN^*$ that will be used to evaluate the performance of a control policy.
\end{itemize}
We often write a model $\mathsf{M}$ as $\mathsf{M}=(\mathbf{X},\mathbf{A},\{\mathbf{A}(x)\}_{x \in \mathbf{X}},Q,\eta,r)$ to specify the notations used for the parameters of this model.

\begin{definition}
We will say that a model $\mathsf{M}=(\mathbf{X},\mathbf{A},\{\mathbf{A}(x)\}_{x\in \mathbf{X}},Q,\eta,r)$ satisfies the \textit{measurability conditions} if the following properties holds.
\begin{List-perso}{M}
\item \label{K-measurability} $\KK$ is a measurable subset of $\big(\XX\times\AA,\bfrak{X}\otimes\bfrak{B}(\AA)\big)$.
\item \label{selector-theta} There exists a function $\theta\in\bcal{L}^{0}_{\AA}(\XX,\bfrak{X})$ whose graph is a subset of $\KK$.
\end{List-perso}
\end{definition}

\bigskip

Let us denote by $\mathbb{D}$ the set of functions $\phi\in \bcal{L}^{0}_{\AA}(\XX,\bfrak{X})$ satisfying $\phi(x)\in\AA(x)$ for any $x\in \XX$ and
by $\mathbb{S}$ (respectively, $\mathbb{C}$) the set of (respectively, finitely supported) stochastic kernels $\varphi$ on $\AA$ given $\XX$ satisfying $\varphi(\AA(x)|x)=1$
for any $x\in \XX$.
We will write $\mathbb{C}_{p}$ for the set of finitely supported kernels in $\mathbb{C}$ of order $p\in\NN^{*}$.

\bigskip

Let us consider now an arbitrary but fixed model $\mathsf{M}=(\mathbf{X},\mathbf{A},\{\mathbf{A}(x)\}_{x\in \mathbf{X}},Q,\eta,r)$ satisfying the measurability conditions.
The space of admissible histories of the controlled process up to time $n\in\NN$ is denoted by $\HH_{n}$. It is defined recursively by 
$\HH_0=\XX\quad\hbox{and}\quad
 \HH_n=\HH_{n-1}\times\AA\times\XX$ for $n\ge1$, all endowed with their corresponding product $\sigma$-algebras.
 A control policy $\pi$ is a sequence $\{\pi_n\}_{n\ge0}$ of stochastic kernels on $\AA$ given $\HH_n$, denoted by $\pi_n(da|h_n)$, such that
$\pi_n(\AA(x_n)|h_n)=1$ for each $n\ge0$ and $h_n=(x_0,a_0,\ldots,x_n)\in\HH_n$. 
The set of all policies is denoted by $\mathbf{\Pi}$. 
\begin{definition}
\label{Def-policies}
We now present several specific families of policies that will be used in the remainder of this work.
\begin{itemize}
\item
A policy $\pi=\{\pi_{n}\}_{n\in\NN}\in \mathbf{\Pi}$ is called a \emph{randomized Markov policy} if there exists a sequence $\{\gamma_{n}\}_{n\in\NN}\in\mathbb{S}$ satisfying
$\pi_{n}(\cdot |h_{n})=\gamma_{n}(\cdot |x_{n})$ for any $h_{n}=(x_{0},a_{0},\ldots,x_{n})\in \mathbf{H}_{n}$ and $n\in \NN$.
\item
A policy $\pi=\{\pi_{n}\}_{n\in\NN}$ is called a \emph{randomized stationary policy} if there exists $\varphi\in\mathbb{S}$ satisfying
$\pi_{n}(\cdot |h_{n})=\varphi(\cdot |x_{n})$ for any $h_{n}=(x_{0},a_{0},\ldots,x_{n})\in \mathbf{H}_{n}$ and $n\in \NN$.
By a slight abuse of notation, we will write $\mathbb{S}$ for the set of all randomized stationary policies.
\item
A policy $\pi=\{\pi_{n}\}_{n\in\NN}$ is called a \emph{chattering Markov policy} if for any $n\in\NN$, 
there exists a finitely supported kernel $\varphi_{n}\in\mathbb{C}$ satisfying
$\pi_{n}(\cdot |h_{n})=\varphi_{n}(\cdot |x_{n})$ for any $h_{n}=(x_{0},a_{0},\ldots,x_{n})\in \mathbf{H}_{n}$ and $n\in \NN$.
\item
A policy $\pi=\{\pi_{n}\}_{n\in\NN}$ is called a \emph{chattering stationary policy} if for any $n\in\NN$, 
there exists a finitely supported kernel $\varphi\in\mathbb{C}$ satisfying
$\pi_{n}(\cdot |h_{n})=\varphi(\cdot |x_{n})$ for $h_{n}=(x_{0},a_{0},\ldots,x_{n})\in \mathbf{H}_{n}$ and $n\in \NN$.
A chattering stationary policy defined by $\varphi\in\mathbb{C}$ will be said of order $p\in\NN^{*}$ if $\varphi$ is of order $p$.
By a slight abuse of notation, we will write $\mathbb{C}_{p}$ for the set of all chattering stationary policies of order $p$ for some $p\in\NN^{*}$.
\item
A policy $\pi=\{\pi_{n}\}_{n\in\NN}\in \mathbf{\Pi}$ is called a \emph{deterministic Markov policy} if there exists a sequence $\{\phi_{n}\}_{n\in\NN}\in\mathbb{S}$ satisfying
$\pi_{n}(\cdot |h_{n})=\delta_{\phi_{n}(x_{n})}(\cdot)$ for any $h_{n}=(x_{0},a_{0},\ldots,x_{n})\in \mathbf{H}_{n}$ and $n\in \NN$.
\item
A policy $\pi=\{\pi_{n}\}_{n\in\NN}\in \mathbf{\Pi}$ is called a \emph{deterministic stationary policy} if there exists $\phi\in\mathbb{D}$ satisfying
$\pi_{n}(\cdot |h_{n})=\delta_{\phi(x_{n})}(\cdot)$ for any $h_{n}=(x_{0},a_{0},\ldots,x_{n})\in \mathbf{H}_{n}$ and $n\in \NN$.
We will identify the set of deterministic stationary policies with the set $\mathbb{D}$. We will therefore use $\mathbb{D}$ to designate these two sets.
\end{itemize}
\end{definition}

For each $\sigma\in\mathbb{S}$, we write 
$\ds Q_{\sigma}(D|x)= \int_{\AA} Q(D|x,a)\sigma(da|x)$
for $x\in\XX$ and $D\in\bfrak{X}$.
The compositions of $Q_\sigma$ with itself are denoted by $Q^t_{\sigma}$ for any $t\ge0$, with the convention that $Q^0_\sigma(\cdot|x)$ is the identity kernel $\mathbb{I}_{\XX}$ on $\XX$.
We also write $\ds g_{\sigma}(x)= \int_{\AA} g(x,a)\sigma(da|x)$ for $x\in\XX$ and a measurable function $g$ from $\XX\times\AA$ to $\RR$ provided this integral is well-defined.
For any Markov randomized policy $\pi=(\pi_{k})_{k\in\NN}$, we write $Q_{\pi}^{(0)}$ for the identity kernel on $\XX$ and 
$Q_{\pi}^{(t)}$ for the product of the transition kernels from step $0$ to $t-1$ for $t\in\NN^{*}$, that is, 
$$Q_{\pi}^{(t)}=\begin{cases}
Q_{\pi_{0}}Q_{\pi_{1}}\cdots Q_{\pi_{t-1}} & \text{for } t\geq 1,\\
\mathbb{I}_{\XX} & \text{for } t=0.
\end{cases}
$$

The canonical space of all possible sample paths of the state-action process is 
$\mathbf{\Omega}=(\XX\times\AA)^{\infty}$ endowed with the product $\sigma$-algebra $\bcal{F}$. 
The coordinate projection functions from $\mathbf{\Omega}$ to the state space $\XX$, the action space $\AA$, and $\HH_n$ for $n\ge0$ are respectively denoted   by $X_{n}$, $A_{n}$, and $H_n$. 
We will refer  to $\{X_{n}\}_{n\in \NN}$ as the state process and $\{A_{n}\}_{n\in \NN}$ as the action  process.
It is a well known result that for every policy $\pi \in \mathbf{\Pi}$ 
there exists a unique probability
measure $\mathbb{P}_{\pi}$ on $(\mathbf{\Omega},\bcal{F})$
such that $ \mathbb{P}_{\pi} (\KK^{\infty})=1$ and such that for every $n\in\NN$, $\Gamma\in \bfrak{X}$, and  $ \Lambda\in \bfrak{B}(\mathbf{A})$ we have 
$\mathbb{P}_{\pi}(X_{0}\in \Gamma)=\eta(\Gamma)$, 
$$\mathbb{P}_{\pi}(X_{n+1}\in \Gamma\mid H_{n},A_{n}) = Q(\Gamma\mid X_{n},A_{n}) \quad\hbox{and}\quad
\mathbb{P}_{\pi}(A_{n}\in \Lambda\mid H_{n})= \pi_{n}(\Lambda\mid H_{n})$$
with
$\mathbb{P}_{\pi}$-probability one.
We will refer to $\mathbb{P}_\pi$ as to a \emph{strategic probability measure}, and we will denote by $\bcal{S}=\{\mathbb{P}_{\pi}\}_{\pi\in\mathbf{\Pi}}\subseteq\bcal{P}(\mathbf{\Omega})$
the set of all strategic probability measures. 
The expectation with respect to  $\mathbb{P}_{\pi}$ is denoted by $\mathbb{E}_{\pi}$. For a $p$-dimensional random variable $Z=(Z_{1},\ldots,Z_{i},\ldots,Z_{p})$ with $p\in\NN^{*}$,
we will write $\mathbb{E}_{\pi}[Z]$ for the vector $\big(\mathbb{E}_{\pi}[Z_{1}],\ldots,\mathbb{E}_{\pi}[Z_{i}],\ldots,\mathbb{E}_{\pi}[Z_{p}]\big)$ when $\mathbb{E}_{\pi}[Z_{i}]$ is well defined for any $i\in\{1,\ldots,p\}$.

\subsection{Absorbing and atomless models}
In this section, we consider an arbitrary but fixed model $\mathsf{M}=(\mathbf{X},\mathbf{A},\{\mathbf{A}(x)\}_{x\in \mathbf{X}},Q,\eta,r)$ satisfying the measurability conditions.
Essentially, we are going to introduce different type of control models such as absorbing and atomless models. To do this, we will need to define the notion of occupation measures associated with a control policy.

\begin{definition}
\label{absorbing}
We say that the control model  $\mathsf{M}$ is $\Delta$-absorbing if $\Delta\in\bfrak{X}$ and the following conditions are satisfied
\begin{enumerate}
\item[$\bullet$] $Q(\Delta|x,a)=1$ and $r(x,a)=0_{d}$ for $(x,a)\in\Delta\times\AA$.
\item[$\bullet$] $\mathbb{E}_{\pi}[T_\Delta]$ is finite for any $\pi\in\mathbf{\Pi}$.
\end{enumerate}
and we say that  it is uniformly $\Delta$-absorbing if, additionally, the following condition holds
\begin{enumerate}
\item[$\bullet$] $\ds \lim_{n\rightarrow\infty} \sup_{\pi\in\mathbf{\Pi}} \sum_{t=n}^\infty\mathbb{P}_{\pi}\{T_\Delta>t\}=0$.
\end{enumerate}
where $T_\Delta:\mathbf{\Omega}\rightarrow\NN\cup\{\infty\}$ is defined by
$T_\Delta(\omega)=\min\{n\ge0: X_{n}(\omega)\in\Delta\},$ where by convention the $\min$ over the empty set is defined as $+\infty$.
\end{definition}
The notion of a uniformly $\Delta$-absorbing MDP was introduced in \cite[Definition 3.6]{piunovskiy19}). 

\bigskip

We define now the occupation measures of a $\Delta$-absorbing control model $\mathsf{M}$. 
\begin{definition}\label{def-occupation-measure}
For a $\Delta$-absorbing control model $\mathsf{M}$, the  occupation measure $\mu_{\pi}\in\MM^{+}(\XX\times\AA)$ of a control policy $\pi\in\mathbf{\Pi}$ is 
$$\mu_{\pi}(\Gamma) = \mathbb{E}_{\pi}\Big[ \sum_{t=0}^\infty \mathbf{I}_{\{T_\Delta>t\}}\cdot\mathbf{I}_{\{(X_t,A_{t})\in \Gamma\}}\Big] \quad\hbox{for $\Gamma\in\bfrak{X}\otimes\mathfrak{B}(\AA)$.}$$
\end{definition}
\begin{remark}
\begin{enumerate}[label=(\alph*)]
\item The measure $\mu_{\pi}$ is indeed a finite measure since $\mu_{\pi}(\XX\times\AA)= \mathbb{E}_{\pi}[T_{\Delta}]<+\infty$
since $\mathsf{M}$ is by assumption $\Delta$-absorbing.
\item Observe that for an arbitrary Markov randomized policy $\pi=\{\pi_{k}\}_{k\in\NN}$, the occupation measure $\mu_{\pi}$ takes the form
\begin{align}
\mu_{\pi} = \sum_{j\in\NN} \eta Q^{(j)}_{\pi} \mathbb{I}_{\Delta^{c}}\otimes \pi_{j}.
\label{Def-occup-meas}
\end{align}
Most of the time we will use this equation without quoting it explicitly.
\end{enumerate}
\end{remark}
We will write $\bcal{O}=\{\mu_{\pi}:\pi\in\mathbf{\Pi}\}$ for the set of occupation measures. We will also consider the following subsets of $\bcal{O}$ denoted by 
$$\bcal{O}_{\Lambda}=\{\mu_{\pi}:\pi\in\Lambda\}$$ 
for the set of policies $\Lambda$ given by $\mathbb{S}$, $\mathbb{D}$ and $\mathbb{C}_{p}$ for some $p\in\NN^{*}$.

\begin{remark}
\begin{enumerate}[label=(\alph*)]
\item Observe that the subset $\bcal{O}$ of the vector space $\MM(\XX\times\AA)$ is convex.
This well known property has been established in the special case of a Borel state space see for example Corollary 4.3 in \cite{feinberg12} and Proposition 3 in \cite{piunovskiy24}.
For the general case of a measurable state space, we can proceed by using similar arguments in showing that the set of strategic measures $ \bcal{S}$ is convex  implying the convexity of $\bcal{O}$.
\item It is also important to recall the following well-known property: $\bcal{O}=\bcal{O}_{\mathbb{S}},$
see for example Theorem 8.1 in \cite{altman99} for a countable state space, Lemma 4.1 in \cite{feinberg12} or Lemma 2 in \cite{piunovskiy24} for a Borel state space
and Proposition 3.3(ii) in \cite{fra-tom2024} for a general measurable state space.
\end{enumerate}
\end{remark}

We will need of the following simple result that says that for two Markov policies which both apply the same fixed control, except at a certain stage where they apply a different control for each of them,
then the corresponding occupation measures are not equal provided that the distribution of the state process at this stage charges the set of states for which these policies are different.
\begin{lemma}
\label{Castaing-Valadier-1}
Assume that the Markov controlled model $\mathsf{M}$ is $\Delta$-absorbing.
Let $\gamma\in\mathbb{S}$ and $N\in\NN$. Consider two Markov policies $\sigma^{1}=\{\sigma^{1}_n\}_{n\ge0}$ and $\sigma^{2}=\{\sigma^{2}_n\}_{n\ge0}$ satisfying
$\sigma^{1}_{k}(\cdot |x)=\sigma^{2}_{k}(\cdot |x)=\gamma(\cdot | x)$ for $k\neq N$. Define $\Gamma=\{x\in\XX : \sigma^{1}_{N}(\cdot |x)\neq \sigma^{2}_{N}(\cdot |x)\}$.
We have $\Gamma\in\bfrak{X}$ and if $\eta Q_{\gamma}^{N}(\Gamma)>0$, then $\mu_{\sigma^{1}} \neq \mu_{\sigma^{2}}$.
\end{lemma}
\begin{proof}
Note that
$\ds \mu_{\sigma^{i}} = \sum_{k=0}^{N-1} \eta Q_{\gamma}^{k} \mathbb{I}_{\Delta^c} \otimes \gamma + \eta Q_{\gamma}^{N} \mathbb{I}_{\Delta^c} \otimes\sigma^{i}_{N}
+\eta Q_{\gamma}^{N}  Q_{\sigma^{i}_{N}} \sum_{j=0}^{\infty} Q_{\gamma}^{j}  \mathbb{I}_{\Delta^c} \otimes \gamma$
for $i=1,2$ from equation \eqref{Def-occup-meas}.
If $\mu_{\sigma^{1}}^{\XX} \neq \mu_{\sigma^{2}}^{\XX}$ then we get the claim. Now, if $\mu_{\sigma^{1}}^{\XX} = \mu_{\sigma^{2}}^{\XX}$ 
then it follows that $\ds \eta Q_{\gamma}^{N}  Q_{\sigma^{1}_{N}} \sum_{j=0}^{\infty} Q_{\gamma}^{j} \mathbb{I}_{\Delta^c}
= \eta Q_{\gamma}^{N} Q_{\sigma^{2}_{N}} \sum_{j=0}^{\infty} Q_{\gamma}^{j} \mathbb{I}_{\Delta^c}$
and so, $\mu_{\sigma^{1}}-\mu_{\sigma^{2}} =  \eta Q_{\gamma}^{N} \mathbb{I}_{\Delta^c} \otimes [\sigma^{1}_{N} - \sigma^{2}_{N}].$
From Theorem 8.10.39 in \cite{bogachev07}, $\MM(\AA)$ is countably separated since the points in $\AA$ can be separated by a countable family of continuous functions on $\AA$
(see Theorem 6.7.7 in \cite{bogachev07}).
Therefore, there exists a countable family of functions $\{h_{n}\}_{n\in\NN}$ on $\AA$ satisfying
$\ds \Gamma=\union_{n\in\NN}\{ x\in\XX : \sigma^{1}_{N} h_{n}(x)\neq \sigma^{2}_{N} h_{n}(x)\}$
and so, $\Gamma\in\bfrak{X}$.
Since $\eta Q^{N}_{\gamma} \mathbb{I}_{\Delta^c} (\Gamma)>0$, there exists $j\in\NN$ such that
either $\eta Q^{N}_{\gamma}\mathbb{I}_{\Delta^c}(\Gamma_{j}^{+}) >0$ or  $\eta Q^{N}_{\gamma}\mathbb{I}_{\Delta^c} (\Gamma_{j}^{-})>0$
where $\Gamma_{j}^{+}=\{x\in\Gamma : \sigma^{1}_{N} h_{j}(x) > \sigma^{2}_{N} h_{j}(x)\}$
and $\Gamma_{j}^{-}=\{x\in\Gamma : \sigma^{1}_{N} h_{j}(x) < \sigma^{2}_{N} h_{j}(x)\}$.
Therefore, it follows that either $\mu_{\sigma^{1}} (\mathbf{I}_{\Gamma_{j}^{+}}h_{j}) >\mu_{\sigma^{2}} (\mathbf{I}_{\Gamma_{j}^{+}}h_{j})$ or
$\mu_{\sigma^{1}} (\mathbf{I}_{\Gamma_{j}^{-}}h_{j}) < \mu_{\sigma^{2}} (\mathbf{I}_{\Gamma_{j}^{-}}h_{j})$
showing the result.
\end{proof}

Now, we introduce the so-called characteristic equation. 
\begin{definition} Suppose that the model $\mathsf{M}$ is $\Delta$-absorbing.
A measure $\mu$ in $\MM^{+}(\XX\times\AA)$ is a solution of the characteristic equations when 
$\mu(\KK^c)=0$ and $\mu^\XX=(\eta+\mu Q)\mathbb{I}_{\Delta^c}$.
We denote by $\bcal{C}$ the set of all solutions of the characteristic equations.
\end{definition}

In the special case of a $\Delta$-absorbing model, the next lemma provides a link between occupation measures and the solutions of the characteristic equation.
We refer the reader to \cite{fra-tom2024}  for a more detailed analysis of these links.
\begin{lemma}
\label{Equivalence-occup-measure-1}
Suppose that the Markov controlled model $\mathsf{M}$ is $\Delta$-absorbing.
Let $\mu\in\bcal{M}(\XX\times\AA)$. Then,
$$  \mu \in \bcal{O}  \iff \Big[\mu \in \bcal{C} \text{ and } \mu^{\XX}\ll \mu^{\XX}_{\pi} \text{ for some } \pi\in\mathbf{\Pi} \Big].$$
\end{lemma}
\begin{proof}
We have only to show the reverse implication since $\bcal{O}\subset\bcal{C}$ by Proposition 3.3(i) in \cite{fra-tom2024}.
From Theorem 3.7(i) in \cite{fra-tom2024} there exists $\sigma\in\mathbb{S}$ such that $\mu=\mu_{\sigma}+\vartheta$ with 
$\vartheta\in\MM^{+}(\XX\times\AA)$ satisfying $\vartheta^\XX=\vartheta Q \mathbb{I}_{\Delta^c}$.
By hypothesis $\mu^{\XX}\ll \mu^{\XX}_{\pi}$, therefore, $\vartheta=0$ according to Proposition 3.8(i) in \cite{fra-tom2024} and so,
$\mu=\mu_{\sigma}\in\bcal{O}$. 
\end{proof}

The performance of a policy $\pi\in\mathbf{\Pi}$ is evaluated by using the following vector-function.
\begin{definition}
Suppose that the model $\mathsf{M}$ is $\Delta$-absorbing. The vector of expected rewards for the policy $\pi\in\mathbf{\Pi}$ denoted by $\mathcal{R}(\pi)$ is given by
$$\mathcal{R}(\pi)=\mathbb{E}_{\pi} \Big[ \sum_{t=0}^\infty r(X_t,A_t)\Big].$$
This defines an application from $\mathbf{\Pi}$ to $\RR^{d}$ called the the performance vector function.
A subset $\Lambda\subset\mathbf{\Pi}$ will be called a sufficient class (family) of policies if $\mathcal{R}(\Lambda)=\mathcal{R}(\mathbf{\Pi})$.
\end{definition}

\begin{remark}
Suppose that the model $\mathsf{M}$ is $\Delta$-absorbing.
For a bounded real-valued measurable function $g$ defined on $\XX\times\AA$ satisfying $g(x,a)=0$ for $x\in\Delta$ and an arbitrary Markov randomized policy $\pi=\{\pi_{k}\}_{k\in\NN}$
we have by using the identity $\mathbb{I}_{\Delta^{c}}g=g$ and equation \eqref{Def-occup-meas}
\begin{align*}
\mathbb{E}_{\pi} \Big[ \sum_{t=0}^\infty g(X_t,A_t)\Big] = \sum_{j\in\NN} \eta Q^{(j)}_{\pi} g_{\pi_{j}}.
\end{align*}
Therefore, $\mathcal{R}(\pi)=\mu_{\pi}(r)$ and $\| \mathcal{R}(\pi) \| \leq \|r\| \mathbb{E}_{\pi}[T_{\Delta}]< +\infty$
for a Markov randomized policy $\pi$.
Most often, these equations will be used implicitly without being mentioned.
\end{remark}

Finally, we present a definition of atomless model which is inspired by the definition introduced in \cite[Definition 2.1]{piunovskiy19} but adapted to the general framework of a measurable state space.
\begin{definition}\label{def-atomless-MDP}
The model $\mathsf{M}$ is called atomless if the measure $\eta\mathbb{I}_{\Delta^c}$ is atomless and if there exists a countably generated sub $\sigma$-algebra 
$\bfrak{X}_{0}$ of $\bfrak{X}$ such that $Q\mathbb{I}_{\Delta^c}(\cdot | x,a)$  is atomless on $\bfrak{X}_{0}$ for any $(x,a)\in\Delta^c\times\AA$.
\end{definition}

\begin{remark}
Note that in the special case where $\XX$ is a Borel space, the previous definition coincides with Definition 2.1 in \cite{piunovskiy19}.
\end{remark}

\begin{remark}
There is no loss of generality to assume that the action space $\mathbf{A}$ is compact. Indeed, if this is not the case, observe that $\mathbf{A}$
is homeomorphic to a Borel subset of a compact metric space denoted by $\hat{\AA}$. 
The sets $\AA(x)$ for $x\in\XX$  are then measurable subsets of $\hat{\AA}$.
The stochastic kernel $Q$ can be extended to a Markov kernel $\hat{Q}$ on $\mathbf{X}$ given $\XX\times\hat{\AA}$ by defining
$\hat{Q}(dy|x,a)=Q(dy|x,a)$ for $(x,a)\in\XX\times\AA$ and $\hat{Q}(dy|x,a)=\eta(dy)$ for $(x,a)\in\XX\times(\hat{\AA}\setminus\AA)$ and the reward function $r$ can be extended to a function $\hat{r}$ given by
$\hat{r}(x,a)=r(x,a)$ for $(x,a)\in\XX\times\AA$ and $\hat{r}(x,a)=0$ for $(x,a)\in\XX\times(\hat{\AA}\setminus\AA)$.
Clearly, $\KK$ is a measurable subset of $\XX\times\hat{\AA}$. Let us write $\hat{\bcal{S}}$ for the set of strategic measures for the model 
$\hat{\mathsf{M}}=(\mathbf{X},\hat{\mathbf{A}},\{\mathbf{A}(x)\}_{x\in \mathbf{X}},\hat{Q},\eta,\hat{r})$.
Observe that for any $\hat{\mathbb{P}}\in \hat{\bcal{S}}$, the restriction of $\hat{\mathbb{P}}$ to $\KK^{\infty}$ coincides with some $\mathbb{P}\in \bcal{S}$ and conversely, for any $\mathbb{P}\in \bcal{S}$,
there exists  $\hat{\mathbb{P}}\in \hat{\bcal{S}}$ such that the restriction of $\hat{\mathbb{P}}$ to $\KK^{\infty}$ coincides with $\mathbb{P}$.
Therefore, the models $\mathsf{M}$ and $\hat{\mathsf{M}}$ can be indentified.
It is also easy to show that if the model $\mathsf{M}$ is (uniformly) $\Delta$-absorbing, so is $\hat{\mathsf{M}}$ and that if $\mathsf{M}$ is atomless, so is  $\hat{\mathsf{M}}$.
In the rest of this paper we will assume that $\mathbf{A}$ is compact.
This extension will allow us to use the results obtained in \cite{fra-tom2024} for $\Delta$-absorbing models.
\end{remark}

Throughout the rest of this paper, we will use a fixed model $\mathsf{M}=(\mathbf{X},\mathbf{A},\{\mathbf{A}(x)\}_{x\in \mathbf{X}},Q,\eta,r)$ satisfying the measurability conditions.
To avoid presenting the results in a too cumbersome manner, we will not repeat in the following that the model $\mathsf{M}$ is assumed to satisfy these conditions.
For technical reasons, we will also have to consider other models which will most often consist of modifying the space of admissible actions of the model $\mathsf{M}$.

\section{Extreme points of constrained sets of occupation measures}
\label{Sec-3}
n this section, we will characterize the extreme points of the set of occupation measures $\mu$ satisfying integral type constraints of the form $\mu(r)=\alpha$ with $\alpha\in\RR^d$
for $\Delta$-absorbing models but also for models that are both $\Delta$-absorbing and atomless.

For $\alpha\in\RR^p$ and a bounded measurable function $g:\XX\times\AA\mapsto\RR^p$ with $p\in\NN^*$, let us denote by $\bcal{H}(g,\alpha)$ the affine hyperplane of the vector space 
$\bcal{M}(\XX\times\AA)$ defined by 
\begin{eqnarray*}
\bcal{H}(g,\alpha)=\{\mu\in\MM(\XX\times\AA): \mu(g)=\alpha\}.
\end{eqnarray*}
We write $\bcal{O}(g,\alpha)=\bcal{O}\inter\bcal{H}(g,\alpha)$ and also,
$\bcal{O}_{\Lambda}(g,\alpha)=\bcal{O}_{\Lambda}\inter\bcal{H}(g,\alpha)$ for $\Lambda$ given by $\mathbb{S}$, $\mathbb{D}$ and $\mathbb{C}_{p+1}$ for $p\in\NN^*$.
In section \ref{Case-absorbing}, we will establish that for a $\Delta$-absorbing model every extreme point of $\bcal{O}(g,\alpha)$ is necessarily given by an occupation measure induced by a chattering stationary policy of order $d+1$, that is,
$\ext(\bcal{O}(g,\alpha))\subset \bcal{O}_{\mathbb{C}_{d+1}}$ (see the Main Theorem in \ref{Dubins}). 
In section \ref{Case-absorbing+atomless}, we will show that the extreme points of $\bcal{O}(g,\alpha)$ are precisely the occupation measures generated by deterministic stationary policies, that is,
$\ext(\bcal{O}(g,\alpha))=\bcal{O}_{\mathbb{D}}(g,\alpha)$ if in addition, the model $\mathsf{M}$ is atomless.

\bigskip

The results of section \ref{Case-absorbing} and \ref{Case-absorbing+atomless} rely heavily on the following preliminary proposition.
The first item states that any stochastic kernel $\gamma\in\mathbb{S}$ can be written as a convex combination of two stochastic kernels $\widebar{\gamma}$ and $\widehat{\gamma}$ in $\mathbb{S}$ such that
the set of points $x\in\XX$ satisfying $\widebar{\gamma}(\cdot | x)\neq \widehat{\gamma}(\cdot | x)$ are, loosely speaking, those for which $\gamma(\cdot | x)$ is not given by a Dirac measure.
To establish it, we will use a general result by Castaing and Valadier (see Theorem IV.14 in \cite{castaing77}) that characterize the extreme points of the set of selectors of a multifunction with convex values.
The second item says that when this set is a $\mu_{\gamma}^{\XX}$-null set then $\mu_{\gamma}$ is induced by a deterministic stationary policy, that is, $\mu_{\gamma}\in\bcal{O}_{\mathbb{D}}$.
\begin{proposition}
\label{Castaing-Valadier-2}
Assume that the Markov controlled model $\mathsf{M}$ is $\Delta$-absorbing.
For $\gamma\in\mathbb{S}$, let us write $\bfrak{X}_{\gamma}$ for the $\mu_{\gamma}^{\XX}$-completion of $\bfrak{X}$
and define $\mathbf{\Lambda}=\big\{x\in\XX : \gamma(\cdot |x)\in \mathbf{\Gamma}(x)\setminus\ext(\mathbf{\Gamma}(x)) \big\}$
where $\mathbf{\Gamma}(x)=\{\lambda\in\bcal{P}(\AA) : \lambda(\AA(x))=1 \}$.
\begin{enumerate}[label=(\alph*)]
\item \label{Existence-gamma1-2} The set $\mathbf{\Lambda}$ is $\bfrak{X}_{\gamma}$-measurable and there exist $\gamma_{1}$ and $\gamma_{2}$ in $\mathbb{S}$ such that
\begin{align}
\label{gamma-property-2}
\gamma+ \frac{1}{2} [\gamma_{1}-\gamma_{2}]\in\mathbb{S} \text{ and }\gamma- \frac{1}{2} [\gamma_{1}-\gamma_{2}]\in\mathbb{S}.
\end{align}
Moreover, $\mathbf{\Xi}=\{x\in\XX : \gamma_{1}(\cdot |x)\neq \gamma_{2}(\cdot |x)\}$ is a $\bfrak{X}$-measurable subset of $\mathbf{\Lambda}$ and $\mathbf{\Lambda}\setminus\mathbf{\Xi}$
is a $\mu_{\gamma}^{\XX}$-null set.
\item \label{Lambda-null-set} If $\mathbf{\Lambda}$ is $\mu_{\gamma}^{\XX}$-null set then $\mu_{\gamma}\in\bcal{O}_{\mathbb{D}}$, that is,
$\mu_{\gamma}=\mu_{\phi}$ for some $\phi\in\mathbb{D}$.
\end{enumerate}
\end{proposition}
\begin{proof}
Let us show the first item.
We will suppose that the spaces $\MM(\AA)$ and $\bcal{P}(\AA)$ are endowed with the weak topology.
Since $\KK \in \bfrak{X}\otimes\bfrak{B}(\AA)$, it follows from Theorem IV.12 in \cite{castaing77} that the graph of the multifunction $\mathbf{\Gamma}: \XX\tto\bcal{P}(\AA)$ defined by
$$\mathbf{\Gamma}(x)=\{\lambda\in\bcal{P}(\AA) : \lambda(\AA(x))=1 \}$$ belongs to $\bfrak{X}\otimes\bfrak{B}(\bcal{P}(\AA))$.
Let us denote by $\mathcal{S}_{\mathbf{\Gamma}}$ the set of $\bfrak{X}$-measurable mappings $\sigma$ from $\XX$ to $\bcal{P}(\AA)$ such that 
$\sigma(\cdot | x) \in \mathbf{\Gamma}(x)$ for any $x\in\XX$.
Recalling Assumption \ref{selector-theta}, $\mathbf{\Gamma}$ has non empty convex values in the Haussdorff locally convex space $\MM(\AA)$.
Observe that the stochastic kernel $\gamma\in\mathbb{S}$ gives rise to a mapping in $\mathcal{S}_{\mathbf{\Gamma}}$ defined by
$x\rightarrow \gamma(\cdot |x)$ and conversely, $\sigma\in\mathcal{S}_{\mathbf{\Gamma}}$ yields to a stochastic kernel in $\mathbb{S}$ given by
$(x,\Lambda)\rightarrow\sigma(x,\Lambda)$ for $x\in\XX$ and $\Lambda\in\bfrak{B}(\AA)$.
This follows easily from Theorem 3.1 in \cite{florescu12} (in this book, it is assumed that $(\XX,\bfrak{X})$ is a complete measurable space but actually the proof of Theorem 3.1 does not use this point).
Depending on the context, we will use $\mathbb{S}$ to refer to control policies and we will use $\mathcal{S}_{\mathbf{\Gamma}}$ when using results from \cite{castaing77}.
According to Theorem IV.14 in  \cite{castaing77} and its proof, the set $\mathbf{\Lambda}=\big\{x\in\XX : \gamma(\cdot |x)\in \mathbf{\Gamma}(x)\setminus\ext(\mathbf{\Gamma}(x)) \big\}$
is a $\bfrak{X}_{\gamma}$-measurable subset of $\XX$ and  there also exist $\varrho_{1}$ and $\varrho_{2}$, $\bfrak{X}_{\gamma}$-measurable selectors of the multifunction $\mathbf{\Gamma}$ satisfying
$\varrho_{1}(\cdot | x)\neq \varrho_{2}(\cdot | x)$ for $x\in\mathbf{\Lambda}$, and $\varrho_{1}(\cdot | x) = \varrho_{2}(\cdot | x)= \gamma(\cdot | x)$ for $x\in\mathbf{\Lambda}^c$, 
and such that $\gamma+ \frac{1}{2} [\varrho_{1}-\varrho_{2}] $ and $\gamma- \frac{1}{2} [\varrho_{1}-\varrho_{2}] $ are $\bfrak{X}_{\gamma}$-measurable selectors of the multifunction $\mathbf{\Gamma}$.
By using Lemma 1.2 in \cite{crauel02} there exist  a $\bfrak{X}$-measurable subset $\Xi$ of $\mathbf{\Lambda}$, two $\bfrak{X}$-measurable mappings $\rho_{1},\rho_{2}$ from $\XX$ to $\bcal{P}(\AA)$  and $\mathcal{N}\in\bfrak{X}$ such that
$ (\mathbf{\Lambda}\setminus\Xi)\union\{x\in\XX : \rho_{1} (\cdot |x)\neq\varrho_{1}(\cdot |x)\}\union\{x\in\XX : \rho_{2} (\cdot |x)\neq\varrho_{2}(\cdot |x)\}\subset \mathcal{N}$
with $\mu^{\XX}_{\gamma}(\mathcal{N})=0$.
Define $\mathbf{\Xi}=\Xi\inter\mathcal{N}^c$. Then $\mathbf{\Xi}\in\bfrak{X}$ and $\mathbf{\Xi}\subset\Xi\subset\mathbf{\Lambda}\subset [\Xi\union\mathcal{N}]=\mathbf{\Xi}\union\mathcal{N}$.
Therefore, $\mathbf{\Lambda}\setminus\mathbf{\Xi}$ is a $\mu_{\gamma}^{\XX}$-null set. Let us define, the stochastic kernels $\gamma_{i}$ on $\AA$ given $\XX$ for $i=1,2$ by $\gamma_{i}(\cdot | x)=\rho_{i}(\cdot | x)$ for $x\in\mathcal{N}^{c}$  and  $\gamma_{i}(\cdot | x)= \gamma(\cdot | x)$ for $x\in\mathcal{N}$. 
Then, $\gamma_{1}$, $\gamma_{2}$ are in $\mathbb{S}$ and satisfy 
\begin{align*}
\gamma_{1}(\cdot | x)\neq \gamma_{2}(\cdot | x) \text{ for } x\in\mathbf{\Xi}, \quad
\gamma_{1}(\cdot | x) = \gamma_{2}(\cdot | x) = \gamma(\cdot | x) \text{ for } x\in\mathbf{\Xi}^c 
\end{align*}
showing that $\mathbf{\Xi}=\{x\in\XX : \gamma_{1}(\cdot |x)\neq \gamma_{2}(\cdot |x)\}\subset\mathbf{\Lambda}$.
Moreover, we have also
\begin{align*}
\gamma+ \frac{1}{2} [\gamma_{1}-\gamma_{2}] \in \mathcal{S}_{\mathbf{\Gamma}} \text{ and } \gamma- \frac{1}{2} [\gamma_{1}-\gamma_{2}] \in \mathcal{S}_{\mathbf{\Gamma}}.
\end{align*}

For the second item, observe that if $\mathbf{\Lambda}$ is $\mu_{\gamma}^{\XX}$-null set then $\mu_{\gamma}^{\XX}(\mathbf{\Xi})=0$.
In this case, we have 
$\gamma(\cdot | x) = \delta_{\{\theta(x)\}}(\cdot) \mathbf{I}_{\mathbf{\Xi}\union\mathcal{N}}(x)+\gamma(\cdot | x) \mathbf{I}_{(\mathbf{\Xi}\union\mathcal{N})^{c}}(x) $
for $\mu_{\gamma}^{\XX}$-almost every $x$ since $\mu_{\gamma}^{\XX}(\mathbf{\Xi}\union\mathcal{N})=0$ where $\theta$ is the $\bfrak{X}$-measurable function from $\XX$ to $\AA$ introduced  in Assumption \ref{selector-theta}.
The extreme points of $\bcal{P}(\AA)$ are given by the Dirac measures from Theorem 15.9 in \cite{aliprantis06}.
Clearly, the mapping $x\mapsto\delta_{\{\theta(x)\}}(\cdot) \mathbf{I}_{\mathbf{\Xi}\union\mathcal{N}}(x)+\gamma(\cdot | x) \mathbf{I}_{(\mathbf{\Xi}\union\mathcal{N})^{c}}(x)$ from $\XX$ to $\bcal{P}(\AA)$
is $\bfrak{X}$-measurable and $\delta_{\{\theta(x)\}}(\cdot) \mathbf{I}_{\mathbf{\Xi}\union\mathcal{N}}(x)+\gamma(\cdot | x) \mathbf{I}_{(\mathbf{\Xi}\union\mathcal{N})^{c}}(x)\in\ext(\mathbf{\Gamma}(x))$
by definition of $\mathbf{\Lambda}$ and recalling $\mathbf{\Lambda}\subset\mathbf{\Xi}\union\mathcal{N}$.
The extreme points of $\bcal{P}(\AA)$ are given by the Dirac measures from Theorem 15.9 in \cite{aliprantis06}.
Observe that $\mathbf{\Gamma}(x)$ is a face of $\bcal{P}(\AA)$ for any $x\in\XX$ and so, $\ext(\mathbf{\Gamma}(x))=\ext(\bcal{P}(\AA))\inter\mathbf{\Gamma}(x)$.
Therefore, there exists a $\bfrak{X}$-measurable mapping $\phi$ from $\XX$ to $\AA$ satisfying $\phi(x)\in\AA(x)$ for any $x\in\XX$ and 
$\mu_{\gamma}(dx,da)=\delta_{\phi(x)}(da)\mu^{\XX}_{\gamma}(dx)$.
Proceeding as for the proof of Proposition 3.3(ii) in \cite{fra-tom2024}, we get easily that $\mu_{\gamma}=\mu_{\phi}$ showing the result.
\end{proof}

\subsection{The case of a general absorbing model}
\label{Case-absorbing}
We start with the following proposition which states that the set of occupation measures of an $\Delta$-absorbing model is linearly closed and linearly bounded, that is, its intersection with each line is closed and bounded
in the natural topology of the line. These are particularly important properties in the study of extremal points of convex sets, see for example \cite{dubins62}.
\begin{proposition}
\label{Linearly-bounded-closed}
Assume that the Markov controlled model $\mathsf{M}$ is $\Delta$-absorbing. The set of occupation of measures $\bcal{O}$ is linearly closed and linearly bounded.
\end{proposition}
\begin{proof}
Let us first show that for $\nu\neq\gamma$ in $\bcal{M}(\XX\times\AA)$, the set $\mathcal{D}=\{u\in\RR : (1-u)\nu+u\gamma \in \bcal{O}\}$ is a bounded interval of $\RR$.
Recalling that $\bcal{O}$ is a convex set, it follows easily that for $s$, $t$ in $\mathcal{D}$, $[s,t]\subset \mathcal{D}$ and so, $\mathcal{D}$ is an interval of $\RR$.
Let $\Gamma$ be a measurable subset of $\XX\times\AA$ such that $\nu(\Gamma)\neq\gamma(\Gamma)$. Then, 
$$\sup_{u\in \mathcal{D}}|u|\leq \Big[\sup_{\mu\in\bcal{O}} \mu(\XX\times\AA) + |\nu(\Gamma)|\Big]/ |\gamma(\Gamma)-\nu(\Gamma)|.$$
By using Lemma 2.4 in \cite{fra-tom2024}, this implies that $\mathcal{D}$ is bounded.  
Finally, let us show that $\mathcal{D}$ is closed. Let us consider $\{u_{n}\}_{n\in\NN}$ a sequence in $\mathcal{D}$ such that $\lim_{n\rightarrow\infty} u_{n}=u\in\RR$.
Then it follows that the measure $\lambda_{u}$ defined by $\lambda_{u}=(1-u)\nu+u\gamma$ lies in $\MM^{+}(\XX\times\AA)$
and also satisfies 
\begin{align*}
\lambda_{u}^{\XX}=\big[\eta+\lambda_{u} Q\big] \mathbb{I}_{\Delta^{c}} \text{ with } \lambda_{u}(\KK^{c})=0
\end{align*}
showing that $\lambda_{u}\in\bcal{C}$. Write $\lambda_{u_{n}}=(1-u_{n})\nu+u_{n}\gamma$ for $n\in\NN$. There is no loss of generality to assume that $\lambda_{u_{n_{0}}}\neq \lambda_{u_{n_{1}}}$
for some $n_{0}$, $n_{1}$ in $\NN$.
Consequently, there exists $\alpha$ and $\beta$ in $\RR$ such that 
$\nu=(1-\alpha) \lambda_{u_{n_{0}}} + \alpha \lambda_{u_{n_{1}}} $ and $\gamma=(1-\beta) \lambda_{u_{n_{0}}} + \beta \lambda_{u_{n_{1}}} $. This implies that
$\lambda_{u} \ll (\lambda_{u_{n_{0}}}+\lambda_{u_{n_{1}}})/2 =\lambda_{\pi}$
for some $\pi\in\mathbf{\Pi}$ since $\bcal{O}$ is convex. Therefore,
$\lambda_{u} \in\bcal{O}$ from Lemma \ref{Equivalence-occup-measure-1} showing that $\mathcal{D}$ is closed.
\end{proof}

We will need the following result established in \cite[Lemma 4.6]{feinberg12} for absorbing models and in \cite[Theorem 1]{piunovskiy24} for models having an absorbing
state in the framework of a Borel state space. It is possible to prove this property in the context of a general measurable state space by using Proposition \ref{Castaing-Valadier-2}.
\begin{proposition}
\label{Extreme-point-Occupation-measure}
Assume that the model $\mathsf{M}$ is $\Delta$-absorbing.
Then, $\ext\big(\bcal{O}\big)=\bcal{O}_{\mathbb{D}}$.
\end{proposition}
\begin{proof}
To get the result, we will first show that $\ext\big(\bcal{O}_{\mathbb{S}}\big)\subset\bcal{O}_{\mathbb{D}}$.
Suppose that $\mu_{\gamma}\in \ext\big(\bcal{O}_{\mathbb{S}}\big)$ for $\gamma\in\mathbb{S}$.
From Proposition \ref{Castaing-Valadier-2}\ref{Existence-gamma1-2}, there exist  $\gamma_{1}$ and $\gamma_{2}$ in $\mathbb{S}$ such that
$\gamma+ \frac{1}{2} [\gamma_{1}-\gamma_{2}]\in\mathbb{S}$ and $\gamma- \frac{1}{2} [\gamma_{1}-\gamma_{2}]\in\mathbb{S}$  
and $\mathbf{\Xi}=\{x\in\mathbf{\Lambda} : \gamma_{1}(\cdot |x)\neq \gamma_{2}(\cdot |x)\}$ is a $\bfrak{X}$-measurable set.
Suppose that $\mu_{\gamma}^{\XX}(\mathbf{\Xi})>0$, then there exists $N$ such that $\eta Q_{\gamma}^{N}(\mathbf{\Xi})>0$.
Therefore, we can define the Markov policies $\sigma^{1}=\{\sigma^{1}_n\}_{n\ge0}$ and $\sigma^{2}=\{\sigma^{2}_n\}_{n\ge0}$ by
$\sigma^{1}_{k}(\cdot |x)=\sigma^{2}_{k}(\cdot |x)=\gamma(\cdot | x)$ for $k\neq N$ and $\sigma^{1}_{N}(\cdot |x)=\gamma(\cdot | x)+ \frac{1}{2} [\gamma_{1}-\gamma_{2}] $ and
$\sigma^{2}_{N}(\cdot |x)=\gamma(\cdot | x)-\frac{1}{2} [\gamma_{1}-\gamma_{2}]$
for $x\in\XX$. By using the expression \eqref{Def-occup-meas} of an occupation measure generated by a Markov policy and the definition of $\sigma^i$ for $i=1,2$, it follows that
$\frac{1}{2}[\mu_{\sigma^1}+\mu_{\sigma^2}] $.
Now observe that $\{x\in\XX : \sigma^{1}_{N}(\cdot |x)\neq \sigma^{2}_{N}(\cdot |x)\}=\{x\in\XX : \gamma_{1}(\cdot |x)\neq \gamma_{2}(\cdot |x)\}=\mathbf{\Xi}$.
From Lemma \ref{Castaing-Valadier-1}, we get that $\mu_{\sigma^{1}}\neq\mu_{\sigma^{2}}$ yielding a contradiction with $\mu_{\gamma}\in \ext\big(\bcal{O}_{\mathbb{S}}\big)$.
Therefore, $\mu_{\gamma}^{\XX}(\mathbf{\Xi})=0$ and so, the set $\mathbf{\Lambda}$ introduced in Proposition \ref{Castaing-Valadier-2} is a $\mu_{\gamma}^{\XX}$-null set and the result follows by 
Proposition \ref{Castaing-Valadier-2}\ref{Lambda-null-set}. The converse inclusion is straightforward.
\end{proof}

The next Theorem is the main result of this subsection.
\begin{theorem}
\label{Dubins}
Assume that the Markov controlled model $\mathsf{M}$ is $\Delta$-absorbing.
Consider a bounded measurable function $g:\XX\times\AA\mapsto\RR^p$ satisfying $g(x,a)=0_{p}$ for $(x,a)\in\Delta\times\AA$ and a constant $\alpha\in\RR^p$ for $p\in\NN^*$.
Every extreme point of $\bcal{O}(g,\alpha)$ is a convex combination of at most $p+1$ extreme points of $\bcal{O}$ and so,
$$\ext\big(\bcal{O}(g,\alpha)\big)\subset \bcal{O}_{\mathbb{C}_{p+1}}(g,\alpha).$$
\end{theorem}
\begin{proof}
Combining Proposition \ref{Linearly-bounded-closed} and the Main Theorem in \cite{dubins62}, we easily get the first part of the claim.
Since by Proposition \ref{Extreme-point-Occupation-measure} $\ext\big(\bcal{O}\big)\subset\bcal{O}_{\mathbb{D}}$, it follows that if $\mu_{\pi} \in\ext\big(\bcal{O}(g,\alpha)\big)$ for some $\pi\in\mathbf{\Pi}$ then
$\mu_{\pi}(g)=\alpha$ and 
$\ds \mu_{\pi}=\sum_{i=1}^{p+1} \alpha_{i} \mu_{\phi_{i}}$ with $\phi_{i}\in\mathbb{D}$ for $i\in\NN_{p+1}^*$ and $(\alpha_{1},\ldots,\alpha_{p+1})\in\bscr{S}_{p+1}$.
By using a slight modification of Lemma 5 in \cite{feinberg2020} to a measurable state space, there exists $\gamma\in \mathbb{C}_{p+1}$ such that $ \mu_{\pi}=\mu_{\gamma}$ and so, 
$\mu_{\gamma}(g)=\alpha$ showing the result.
\end{proof}

\subsection{The case of an absorbing and atomless model}
\label{Case-absorbing+atomless}
To establish the main result of this subsection, that is, $\ext(\bcal{O}(g,\alpha))=\bcal{O}_{\mathbb{D}}(g,\alpha)$ we will need the following technical lemma.
\begin{lemma}
\label{Tech-Castaing-Valadier-3}
Assume that the Markov controlled model $\mathsf{M}$ is $\Delta$-absorbing.
Consider $N\in\NN$ and a bounded function $g\in \bcal{L}^{0}_{\RR^{p}}(\XX\times\AA)$ satisfying $g(x,a)=0_{p}$ for $(x,a)\in\Delta\times\AA$, $p\in\NN^*$.
For $\gamma$, $\pi^{1}$ and $\pi^{2}$ in $\mathbb{S}$, the $p$-dimensional signed measure $\nu$ on $\XX$ defined by 
$$\nu(dx)= \Big[ g_{\pi}(x)+Q_{\pi} \sum_{k\in\NN} Q_{\gamma}^{k}  g_{\gamma} (x)\Big]  \: \eta Q_{\gamma}^{N} \mathbb{I}_{\Delta^c} (dx)$$
is finite and absolutely continuous with respect to $\eta Q_{\gamma}^{N}\mathbb{I}_{\Delta^c}$ where $g_{\pi}=g_{\pi^{1}}-g_{\pi^{2}}$ and $Q_{\pi}=Q_{\pi^{1}}-Q_{\pi^{2}}$.
\end{lemma}
\begin{proof}
Consider the Markov policy $\sigma^{i}=\{\sigma^{i}_n\}_{n\ge0}$ for $i=1,2$ defined by
$\sigma^{i}_{k}(\cdot |x)=\gamma(\cdot | x)$ for $k\neq N$ and 
$\sigma^{i}_{N}(\cdot |x) = \pi^{i}(\cdot | x)$, for $x\in\XX$. 
Since $\mathsf{M}$ is $\Delta$-absorbing
\begin{align} 
\label{reward-expression-1}
\mathbb{E}_{\sigma^{i}} \Big[\sum_{j\in\NN} h(X_{j},A_{j})\Big] 
& = \sum_{j=0}^{N-1} \eta Q^{j}_{\gamma} h_{\gamma} + \eta Q^{N}_{\gamma} \mathbb{I}_{\Delta^c} \big[ h_{\pi^{i}} 
+  Q_{\pi^{i}} \sum_{k\in\NN} Q_{\gamma}^{k} \mathbb{I}_{\Delta^c} h_{\gamma} \big] <+\infty
\end{align}
for $N\in\NN$ and $h(x,a)=|g(x,a)|$.
Consequently, the components of the $p$-dimensional functions $g_{\pi^{i}}  + Q_{\pi^{i}} \sum_{k\in\NN} Q_{\gamma}^{k}  g_{\gamma}$ for $i=1,2$
are $\eta Q^{N}_{\gamma} \mathbb{I}_{\Delta^c}$-measurable, see Definition 2.1.10 in \cite{bogachev07} and so, are the components of the function 
$g_{\pi}  + Q_{\pi} \sum_{k\in\NN} Q_{\gamma}^{k}  g_{\gamma}$.
By using again inequality \eqref{reward-expression-1}, it is also $\eta Q^{N}_{\gamma} \mathbb{I}_{\Delta^c}$-integrable. Therefore, the $p$-dimensional signed measure $\nu$ defined by
$$\nu(dx)= \Big[ g_{\pi}(x)+Q_{\pi} \sum_{k\in\NN} Q_{\gamma}^{k}  g_{\gamma} (x)\Big]  \: \eta Q_{\gamma}^{N} \mathbb{I}_{\Delta^c} (dx)$$
is of finite total variation and absolutely continuous with respect to $\eta Q_{\gamma}^{N}\mathbb{I}_{\Delta^c}$ showing the result.
\end{proof}
The next Theorem is the main result of this subsection.
\begin{theorem}
\label{Castaing-Valadier}
Assume that the Markov controlled model $\mathsf{M}$ is $\Delta$-absorbing and atomless.
Consider a bounded function $g\in \bcal{L}^{0}_{\RR^{p}}(\XX\times\AA)$ satisfying $g(x,a)=0_{p}$ for $(x,a)\in\Delta\times\AA$ and a constant $\alpha\in\RR^p$ for $p\in\NN^{*}$.
Then, we have
$$\ext\big(\bcal{O}(g,\alpha)\big) = \bcal{O}_{\mathbb{D}}(g,\alpha).$$
\end{theorem}
\begin{proof}
Let us first show that $\ext\big(\bcal{O}(g,\alpha)\big) \subset \bcal{O}_{\mathbb{D}}(g,\alpha)$.
There is no loss of generality to suppose that there exists $\mu_{\gamma}\in\bcal{O}(g,\alpha)$ for $\gamma\in\mathbb{S}$, otherwise the result is straightforward.
From Proposition \ref{Castaing-Valadier-2}\ref{Existence-gamma1-2}, there exist  $\gamma_{1}$ and $\gamma_{2}$ in $\mathbb{S}$ such that
$\gamma \pm \frac{1}{2} [\gamma_{1}-\gamma_{2}]\in\mathbb{S}$  and $\mathbf{\Xi}=\{x\in\XX: \gamma_{1}(\cdot |x)\neq \gamma_{2}(\cdot |x)\}$ is a $\bfrak{X}$-measurable set.

Suppose that $\mu_{\gamma}^{\XX}(\mathbf{\Xi})>0$, then there exists $N\in\NN$ such that $\eta Q_{\gamma}^{N}(\mathbf{\Xi})>0$.
Write $\varphi= \frac{1}{2} [\gamma_{1}-\gamma_{2}]$.
According to Lemma \ref{Tech-Castaing-Valadier-3}, the $p$-dimensional signed measure $\nu$ on $\XX$ defined by 
$$\nu(dx)= \Big[ g_{\varphi}(x)+Q_{\varphi}\sum_{k\in\NN} Q_{\gamma}^{k} g_{\gamma} (x) \Big]  \: \eta Q_{\gamma}^{N} \mathbb{I}_{\Delta^c} (dx)$$
is of finite total variation and absolutely continuous with respect to $\eta Q_{\gamma}^{N}\mathbb{I}_{\Delta^c}$.
By hypothesis, the probability $\eta Q_{\gamma}^{N}\mathbb{I}_{\Delta^c}$ is atomless and so, 
by Lemma 6 in \cite{halmos48} this measure is convex as defined on page 417 in \cite{halmos48}.
From Lemma 3 in \cite{halmos48},  there exists a measurable subset $\mathbf{\Xi}_{*}$
of $\mathbf{\Xi}$ satisfying
\begin{align}
\label{Def-atom-r-Qrbar}
\int_{\mathbf{\Xi}_{*}} \Big[g_{\varphi}(x)+Q_{\varphi} \sum_{k\in\NN} Q_{\gamma}^{k} g_{\gamma} (x) \Big]  \: \eta Q_{\gamma}^{N} \mathbb{I}_{\Delta^c} (dx)
& =\frac{1}{2}\int_{\mathbf{\Xi}} \Big[g_{\varphi}(x)+Q_{\varphi} \sum_{k\in\NN} Q_{\gamma}^{k}  g_{\gamma} (x)\Big] \: \eta Q_{\gamma}^{N} \mathbb{I}_{\Delta^c} (dx),
\end{align}

From equation \eqref{gamma-property-2} we can define the Markov policies $\sigma^{1}=\{\sigma^{1}_n\}_{n\ge0}$ and $\sigma^{2}=\{\sigma^{2}_n\}_{n\ge0}$ by
$\sigma^{1}_{k}(\cdot |x)=\sigma^{2}_{k}(\cdot |x)=\gamma(\cdot | x)$ for $k\neq N$ and 
\begin{align}
\label{Def-sigma1-N}
\sigma^{1}_{N}(\cdot |x) & =\gamma(\cdot | x)+\varphi(\cdot | x) \mathbf{I}_{\mathbf{\Xi}_{*}}(x) - \varphi(\cdot | x) \mathbf{I}_{\mathbf{\Xi}\setminus\mathbf{\Xi}_{*}}(x),
 \\
\label{Def-sigma2-N}
\sigma^{2}_{N}(\cdot |x) & =\gamma(\cdot | x)-\varphi(\cdot | x) \mathbf{I}_{\mathbf{\Xi}_{*}}(x) + \varphi(\cdot | x) \mathbf{I}_{\mathbf{\Xi}\setminus\mathbf{\Xi}_{*}}(x),
\end{align}
for $x\in\XX$.
Observe that
\begin{align}
\label{Def-mu-sigma-i}
\mu_{\sigma^{i}} = \sum_{k=0}^{N-1} \eta Q_{\gamma}^{k} \mathbb{I}_{\Delta^c} \otimes \gamma + \eta Q_{\gamma}^{N} \mathbb{I}_{\Delta^c} \otimes\sigma^{i}_{N}
+\eta Q_{\gamma}^{N}  Q_{\sigma^{i}_{N}} \sum_{j=0}^{\infty} Q_{\gamma}^{j}  \mathbb{I}_{\Delta^c} \otimes \gamma
\end{align}
for $i=1,2$.

Combining the definition of $\sigma^{i}_{N}$ (see equations \eqref{Def-sigma1-N}-\eqref{Def-sigma2-N}) and the expression of the occupation measures $\mu_{\sigma^{i}}$
given in equation \eqref{Def-mu-sigma-i} for $i=1,2$ we obtain
\begin{align}
\mu_{\gamma}=\frac{1}{2} \big( \mu_{\sigma^{1}} + \mu_{\sigma^{2}} \big).
\label{convexcombination}
\end{align}
We also have 
\begin{align}
\mu_{\sigma^{i}}\in \bcal{O}(g,\alpha)
\label{equality-reward}
\end{align}
for $i=1,2$. Indeed, equation \eqref{Def-atom-r-Qrbar} leads to
\begin{align*}
\int_{\mathbf{\Xi}_{*}} \Big[ g_{\varphi}(x)+Q_{\varphi} \sum_{k\in\NN} Q_{\gamma}^{k}  g_{\gamma} (x) \Big] \:  & \eta Q_{\gamma}^{N} \mathbb{I}_{\Delta^c}(dx)
 =\int_{\mathbf{\Xi}\setminus\mathbf{\Xi}_{*}} \Big[g_{\varphi}(x)+Q_{\varphi} \sum_{k\in\NN} Q_{\gamma}^{k} g_{\gamma} (x) \Big] \: \eta Q_{\gamma}^{N} \mathbb{I}_{\Delta^c}(dx).
\end{align*}
Recalling the definition of $\sigma^{i}_{N}$ for $i=1,2$ (see equations \eqref{Def-sigma1-N}-\eqref{Def-sigma2-N}), this implies that
\begin{align*}
\int_{\XX} \Big[g_{\sigma^{i}_{N}}(x)+Q_{\sigma^{i}_{N}} \sum_{k\in\NN} Q_{\gamma}^{k} g_{\gamma} (x) \Big] \: \eta Q_{\gamma}^{N}\mathbb{I}_{\Delta^c}(dx)
=\int_{\XX} \Big[g_{\gamma}(x)+Q_{\gamma}\sum_{k\in\NN} Q_{\gamma}^{k} g_{\gamma} (x)\Big] \: \eta Q_{\gamma}^{N}\mathbb{I}_{\Delta^c}(dx).
\end{align*}
From equation \eqref{Def-mu-sigma-i}, we thus obtain $\mu_{\sigma^{1}}(g)= \mu_{\sigma^{2}}(g)= \mu_{\gamma}(g)=\alpha$ yielding \eqref{equality-reward}.
Now observe that $\{x\in\XX : \sigma^{1}_{N}(\cdot |x)\neq \sigma^{2}_{N}(\cdot |x)\}=\{x\in\XX : \gamma_{1}(\cdot |x)\neq \gamma_{2}(\cdot |x)\}=\mathbf{\Xi}$.
From Lemma \ref{Castaing-Valadier-1}, we get that 
\begin{align}
\mu_{\sigma^{1}}\neq\mu_{\sigma^{2}}.
\label{mu1-neq-mu2}
\end{align}
Finally and recalling equations  \eqref{convexcombination}, \eqref{equality-reward} and \eqref{mu1-neq-mu2},
we have established that  if $\mu_{\gamma}^{\XX}(\mathbf{\Xi})>0$ for $\mu_{\gamma} \in \bcal{O}(g,\alpha)$ then 
there exist $\mu_{\sigma^{1}}$ and $\mu_{\sigma^{2}}$ in $\bcal{O}(g,\alpha)$ satisfying $\mu_{\gamma}=\frac{1}{2} \big( \mu_{\sigma^{1}} + \mu_{\sigma^{2}} \big)$ and also $\mu_{\sigma^{1}} \neq \mu_{\sigma^{2}}$ showing that $\mu_{\gamma}$ is not an extreme point of 
$\bcal{O}(g,\alpha)$. Therefore, if $\mu_{\gamma}$ is an extreme point of $\bcal{O}(g,\alpha)$, then $\mu_{\gamma}^{\XX}(\mathbf{\Xi})=0$ and so, the set $\mathbf{\Lambda}$ introduced in Proposition \ref{Castaing-Valadier-2} is a
$\mu_{\gamma}^{\XX}$-null set and the result follows by  Proposition \ref{Castaing-Valadier-2}\ref{Lambda-null-set}.

Under the assumption that the model $\mathsf{M}$ is $\Delta$-absorbing, we can get the reverse inclusion, that is,
$\bcal{O}_{\mathbb{D}}(g,\alpha) \subset \ext\big(\bcal{O}(g,\alpha)\big)$
Let us consider $\mu_{\Phi}\in \bcal{O}_{\mathbb{D}}(g,\alpha)$ for $\Phi\in\mathbb{D}$. Since $\ext(\bcal{O})=\bcal{O}_{\mathbb{D}}$ (see Proposition \ref{Extreme-point-Occupation-measure}), we have
that $\mu_{\Phi}\in \ext(\bcal{O})\inter\bcal{H}(g,\alpha) \subset \ext(\bcal{O}\inter\bcal{H}(g,\alpha))$ showing the claim.
\end{proof}

\section{Sufficiency of chattering and deterministic Markov policies}
\label{Sec-4}
The main objective of this section is to show that for a uniformly $\Delta$-absorbing model $\mathsf{M}$, the set of chattering Markov policies of order $2d+1$ is a sufficient class of control strategies. In other words, for any control policy $\pi\in\mathbf{\Pi}$, there exists a chattering Markov policy $\gamma$ of order $2d+1$ satisfying $\mathcal{R}(\gamma)=\mathcal{R}(\pi)$.
To our best knowledge, this result is new and it will be reinforced in the following section by showing the stronger result that the set of chattering stationary policies of order $d+1$ is in fact sufficient.
If in addition the model is atomless then we will show that the set of deterministic Markov policy is sufficient.
As already mentioned in the introduction, the sufficiency of deterministic Markov policies for an atomless model with Borel state space is well known. The reader is referred  to the references \cite{piunovskiy00ef,piunovskiy02ef} where the authors obtained this type of results in a very general context by considering a vector performance functional having the form of total rewards (\textit{i.e.} the model under consideration are not necessarily absorbing).
Nevertheless, we believe it is worthwhile to propose this new statement, which deals with general state space,
but also because we are proposing a new direct approach to prove this result, based on the Young measure theory.

We start this section by showing that for an arbitrary Markov randomized policy $\pi\in\mathbf{\Pi}$ and integer $t\in\NN$, we can replace the control $\pi_{t}$ at step $t$ by a finitely supported control $\gamma\in\mathbb{C}_{2p+1}$ without 
changing both the values of the vectors of expected rewards at step $t$ and the expected total rewards by assuming $\mathsf{M}$ is $\Delta$-absorbing. If in addition $\mathsf{M}$ is atomless then the control $\gamma$ can be chosen in $\mathbb{D}$ (\textit{i.e.} a deterministic control).
To get these two results we will apply the finite support equivalence Theorem and the Dirac equivalence Theorem, two well known results coming from the Young measure theory, see Theorem 8.2 and Theorem 9.2 in \cite{balder95}.
\begin{proposition}
\label{Theo-invariance}
Suppose that $\mathsf{M}$ is $\Delta$-absorbing.
Consider a bounded function $g\in \bcal{L}^{0}_{\RR_{+}^{p}}(\XX\times\AA)$ for $p\in\NN^{*}$ satisfying $g(x,a)=0_{p}$ for $(x,a)\in\Delta\times\AA$, an arbitrary Markov randomized policy $\pi=\{\pi_{k}\}_{k\in\NN}$
with $\pi_{k}\in\mathbb{S}$ and an integer $t\in\NN$.
Then the following properties hold. 
\begin{enumerate}[label=(\alph*)]
\item \label{Theo-invariance-a} There exists a policy $\tilde{\pi}=\{\tilde{\pi}_{k}\}_{k\in\NN}$
defined by $\tilde{\pi}_{j}=\pi_{j}$ for $j\neq t$ and $\tilde{\pi}_{t}(da|x)=\gamma(da|x)$
for some finitely supported kernel $\gamma\in\mathbb{C}_{2p+1}$ of order $2p+1$ satisfying
\begin{align}
\mathbb{E}_{\tilde{\pi}} [g(X_{t},A_{t})] = \mathbb{E}_{\pi} [g(X_{t},A_{t})] \quad \text{and} \quad  \mathbb{E}_{\tilde{\pi}} \Big[\sum_{j\in\NN} g(X_{j},A_{j})\Big]
= \mathbb{E}_{\pi} \Big[\sum_{j\in\NN} g(X_{j},A_{j} ) \Big] .
\label{Eq-Theo-invariance}
\end{align}
\item \label{Theo-invariance-b} If in addition $\mathsf{M}$ is atomless then there exists $\phi\in\mathbb{D}$ for which the policy $\tilde{\pi}=\{\tilde{\pi}_{k}\}_{k\in\NN}$
defined by $\tilde{\pi}_{j}=\pi_{j}$ for $j\neq t$ and $\tilde{\pi}_{t}(da|x)=\delta_{\phi(x)}(da)$ satisfies the equation \eqref{Eq-Theo-invariance}.
\end{enumerate}
\end{proposition}
\begin{proof}
For a fixed Markov randomized policy $\pi=\{\pi_{k}\}_{k\in\NN}$ and a fixed integer $t\in\NN$, we have that the function $g_{t}^{\pi}$ defined on $\XX$ by
$$g_{t}^{\pi}(x)=  g_{\pi_{t+1}}(x)+Q_{\pi_{t+1}} g_{\pi_{t+2}}(x)+Q_{\pi_{t+1}}Q_{\pi_{t+2}} g_{\pi_{t+3}}(x)+\cdots \: $$
is measurable and takes values in $\widebar{\RR}_{+}^{p}$. Therefore, the function $Qg_{t}^{\pi}$ is also measurable and takes values in $\widebar{\RR}_{+}^{p}$.
Now consider an arbitrary $\kappa\in\mathbb{S}$ and the Markov randomized policy $\sigma=\{\sigma_{k}\}_{k\in\NN}$ defined by $\sigma_{k}=\pi_{k}$ for $k\neq t$ and $\sigma_{t}=\kappa$.
By Fubini's Theorem and since $\mathsf{M}$ is $\Delta$-absorbing we have 
\begin{align} 
\label{reward-expression}
\mathbb{E}_{\sigma} \Big[\sum_{j\in\NN} g(X_{j},A_{j})\Big] & 
= \sum_{j=0}^{t-1} \eta Q^{(j)}_{\pi} g_{\pi_{j}} + \eta Q^{(t)}_{\pi} g_{\kappa}  + \eta Q^{(t)}_{\pi} Q_{\kappa} g_{t}^{\pi} <+\infty.
\end{align}
Consequently, $\eta Q^{(t)}_{\pi}\otimes\kappa (\Gamma) = 0 $ where $\Gamma=\big\{(x,a)\in\XX\times\AA : | Qg_{t}^{\pi}(x,a) | =+\infty \big\}$ for any $\kappa\in\mathbb{S}$.
Let us write 
$$h(x,a)=Qg_{t}^{\pi}(x,a) \mathbf{I}_{\Gamma^c}(x,a).$$

To show the first item, we apply Theorem \ref{finitely-supported-young-measure} to the finite measure $\eta Q^{(t)}_{\pi}$, the kernel $\pi_{t}\in\mathbb{S}$ and the $\RR^{2p}$-valued measurable function 
$(g,h)$ to get a stochastic kernel $\gamma\in\mathbb{C}_{2p+1}$ satisfying
\begin{align}
g_{\pi_{t}} (x) = g_{\gamma}(x) \text{ and } h_{\pi_{t}} (x) = h_{\gamma}(x), 
\label{exist-Deter-eq1-2}
\end{align}
for $\eta Q^{(t)}_{\pi}$-almost every $x$.
For the policy $\tilde{\pi}$ as defined in the statement, the first equality in \eqref{exist-Deter-eq1-2} yields $\mathbb{E}_{\pi} [g(X_{t},A_{t})] = \mathbb{E}_{\tilde{\pi}} [g(X_{t},A_{t})]$.
Now, the second equality in \eqref{exist-Deter-eq1-2} leads to $\eta Q^{(t)}_{\pi} h_{\pi_{t}} = \eta Q^{(t)}_{\pi} h_{\gamma} $ and so, we have 
$\eta Q^{(t)}_{\pi}Q_{\pi_{t}}g_{t}^{\pi} = \eta Q^{(t)}_{\pi}Q_{\gamma}g_{t}^{\pi} $ since $\eta Q^{(t)}_{\pi}\otimes\pi_{t} (\Gamma) =\eta Q^{(t)}_{\pi}\otimes\gamma (\Gamma) = 0 $.
Using Fubini's Theorem and  equation \eqref{reward-expression},  we have
\begin{align*} 
\mathbb{E}_{\pi} \Big[\sum_{j\in\NN} g(X_{j},A_{j})\Big] & = \sum_{j=0}^{t-1} \eta Q^{(j)}_{\pi}  g_{\pi_{j}} + \eta Q^{(t)}_{\pi}  g_{\gamma}
+ \eta Q^{(t)}_{\pi}  Q_{\gamma} g_{t}^{\pi} = \mathbb{E}_{\tilde{\pi}} \Big[\sum_{j\in\NN} g(X_{j},A_{j})\Big]
\end{align*}
by definition of $\tilde{\pi}$, giving the first part of the claim.

\bigskip

To prove point $(b)$, we will proceed in the same way as for point $(a)$. We will, therefore, present a less detailed proof than for point $(a)$.
Under the assumption that $\mathsf{M}$ is atomless, the measure $\eta Q^{(t)}_{\pi} \mathbb{I}_{\Delta^c}$ is atomless.
Applying Theorem 9.2 in \cite{balder95}  (the Dirac equivalence Theorem for Young measures) to the $\RR^{2p+1}$-valued measurable function
$(x,a)\mapsto \big(\mathbf{I}_{\KK ^c}(x,a),g(x,a),h(x,a)\big)$, there exists a $\bfrak{X}$-measurable function $\phi: \XX \mapsto \AA$ satisfying
\begin{align}
\eta Q^{(t)}_{\pi} \mathbb{I}_{\Delta^c}  \otimes \pi_{t} (\KK ^c) & = \eta Q^{(t)}_{\pi} \mathbb{I}_{\Delta^c} \otimes \delta_{\phi} (\KK ^c)
\label{A-selector-constraint}
\\
\eta Q^{(t)}_{\pi} \mathbb{I}_{\Delta^c} g_{\pi_{t}} & = \eta Q^{(t)}_{\pi}\mathbb{I}_{\Delta^c}  g_{\phi}, 
\label{exist-Deter-eq1bis}
\\
\eta Q^{(t)}_{\pi}\mathbb{I}_{\Delta^c} Q_{\pi_{t}} g_{t}^{\pi} & = \eta Q^{(t)}_{\pi}\mathbb{I}_{\Delta^c} Q_{\phi} g_{t}^{\pi}, 
\label{exist-Deter-eq2bis}
\end{align}
where $\delta_{\phi}:\XX\mapsto\bcal{P}(\AA)$ is given by $\delta_{\phi(x)}(da)$ for any $x\in\XX$.
Equation \eqref{A-selector-constraint} implies $\phi(x)\in\AA(x)$ for $\eta Q^{(t)}_{\pi}\mathbb{I}_{\Delta^c}$ almost every $x$.
Recalling Assumption \ref{selector-theta}, 
we can conclude that there exists a mapping from $\XX$ to $\AA$ still denoted by $\phi$ satisfying $\phi(x)\in\AA(x)$ for any $x\in\XX$ and also equations \eqref{A-selector-constraint}-\eqref{exist-Deter-eq2bis}.
Starting from equation \eqref{reward-expression} and proceeding now exactly as for the proof of the first item, we easily obtain the result.
\end{proof}

\begin{theorem}
\label{deterministic-Markov-policy}
Suppose $\mathsf{M}$ uniformly $\Delta$-absorbing. Then,
\begin{enumerate}[label=(\alph*)]
\item \label{deterministic-Markov-policy-a} For any policy $\pi\in\mathbf{\Pi}$, there exists a chattering Markov policy $\gamma=\{\gamma_{n}\}_{n\in\NN}$ with  $\gamma_{n}\in\mathbb{C}_{2d+1}$ and 
satisfying $\mu_{\pi}(r)=\mu_{\gamma}(r)$.
\item \label{deterministic-Markov-policy-b} If in addition $\mathsf{M}$ is atomless then for any $\pi\in\mathbf{\Pi}$ there exists a deterministic Markov policy $\Phi$ such that $\mu_{\pi}(r)=\mu_{\Phi}(r)$. 
\end{enumerate}
\end{theorem}
\begin{proof} 
Let us write $g$ for the $\RR^{2d}$-valued function defined on $\XX\times\AA$ by $g(x,a)=(r^{+}(x,a),r^{-}(x,a))$ where $r^{+}$ and $r^{-}$ are the positive and negative parts of $r$.
Let us show that for any randomized stationary policy $\pi\in\mathbb{S}$, there exists a chattering Markov policy $\gamma$ satisfying
\begin{align*} 
\mathbb{E}_{\pi} \Big[  \sum_{j=0}^{\infty}   g(X_{j},A_{j})\Big]  = \mathbb{E}_{\gamma} \Big[ \sum_{j=0}^{\infty}g(X_{j},A_{j})\Big].
\end{align*}
This will give the result since $\bcal{O}=\bcal{O}_{\mathbb{S}}$ according to Proposition 3.3(ii) in \cite{fra-tom2024}.
Applying Proposition \ref{Theo-invariance}\ref{Theo-invariance-a} and proceeding by induction, it can be shown that there exists
a sequence $\{\gamma_{k}\}_{k\in\NN}$ in $\mathbb{D}$ satisfying 
\begin{align} 
\label{reward-constant}
\mathbb{E}_{\pi} \Big[\sum_{j\in\NN} g(X_{j},A_{j})\Big] &  = \mathbb{E}_{\gamma^{t}} \Big[\sum_{j\in\NN} g(X_{j},A_{j})\Big]
\end{align}
for any $t\in\NN$ and where the policy $\gamma^{t}=\{\gamma_{k}^{t}\}_{k\in\NN}$ is defined by $\gamma_{k}^{t}=\gamma_{k}$ for $k\leq t$ and $\gamma_{k}^{t}=\pi$ for $k>t$.
Let $\gamma$ be the chattering Markov policy given by $\{\gamma_{k}\}_{k\in\NN}$. Then, we have $\mathbb{E}_{\gamma} [g(X_{s},A_{s})] =\mathbb{E}_{\gamma^{t}} [g(X_{s},A_{s})] $ for any $s\leq t$ and so,
\begin{align} 
\label{finite-time-reward-equal}
\mathbb{E}_{\gamma} \Big[\sum_{j=0}^{t} g(X_{j},A_{j})\Big] & = \mathbb{E}_{\gamma^{t}} \Big[\sum_{j=0}^{t} g(X_{j},A_{j})\Big].
\end{align}
Combining equations \eqref{reward-constant}-\eqref{finite-time-reward-equal}, we have for any $t\in\NN$
\begin{align*} 
\Big| \mathbb{E}_{\pi} \Big[\sum_{j\in\NN} g(X_{j},A_{j})\Big] - \mathbb{E}_{\gamma} \Big[\sum_{j\in\NN} g(X_{j},A_{j})\Big] \Big| 
&= \Big| \mathbb{E}_{\gamma^{t}} \Big[\sum_{j >t} g(X_{j},A_{j})\Big] - \mathbb{E}_{\gamma} \Big[\sum_{j>t} g(X_{j},A_{j})\Big] \Big|\\
&\leq 2 \sup_{\rho\in\mathbf{\Pi}} \mathbb{E}_{\rho} \Big[\sum_{j>t} |g(X_{j},A_{j}) | \Big]\
\end{align*}
Since $\mathsf{M}$ uniformly $\Delta$-absorbing, this shows the first claim.

The second statement can be proved in exactly the same way, except that item \ref{Theo-invariance-b} of Proposition \ref{Theo-invariance} is used instead of \ref{Theo-invariance-a}
to prove this second assertion.
\end{proof}

\section{Sufficiency of chattering and deterministic stationary policies}
\label{Sec-5}
In this last section we will establish the sufficiency of the families of policies $\mathbb{C}_{d+1}$ and $\mathbb{D}$ under an appropriate set of hypotheses.
We will start by showing the following technical result.
\begin{lemma}
\label{finite-compactness}
Suppose $\mathsf{M}$ uniformly $\Delta$-absorbing.
Consider $\{\phi_{k}\}_{k=1,\ldots,N}$ in $\mathbb{D}$ for $N\in\NN^{*}$, a constant $\alpha\in\RR^p$ and
a bounded function $g\in \bcal{L}^{0}_{\RR_{+}^{p}}(\XX\times\AA)$ for $p\in\NN^{*}$ satisfying $g(x,a)=0_{p}$ for $(x,a)\in\Delta\times\AA$.
Let us write 
\begin{align}
\label{Def-Otilde}
\widetilde{\bcal{O}}=\{\mu\in \bcal{O}: \mu(\widetilde{\KK}^{c})=0\}
\end{align}
where $\widetilde{\KK}=\big\{(x,a)\in\XX\times\AA : a\in\{\phi_{1}(x),\ldots,\phi_{N}(x)\}\big\}$.
Then, $\widetilde{\KK}\in\bfrak{X}\otimes\bfrak{B}(\AA)$ and $\widetilde{\bcal{O}}$ is a $ws$-compact face of $\bcal{O}$.
If the set $ \widetilde{\bcal{O}} \inter \bcal{H}(g,\alpha)$ is nonempty then,
$\ext \big( \bcal{O}(g,\alpha) \big) \neq \emptyset.$
\end{lemma}
\begin{proof}
Clearly, $\widetilde{\bcal{O}}$ is a face of $\bcal{O}$. Clearly, $\widetilde{\bcal{O}}$ corresponds to the set of occupation measures of the controlled model
$\widetilde{\mathsf{M}}=(\mathbf{X},\mathbf{A},\{\widetilde{\mathbf{A}}(x)\}_{x\in \mathbf{X}},Q,\eta,r)$ where
$\widetilde{\AA}(x)=\{\phi_{1}(x),\ldots,\phi_{N}(x)\}$.

The model $\mathsf{M}$ being, by hypothesis, uniformly $\Delta$-absorbing, it follows that $\widetilde{\mathsf{M}}$ is also uniformly $\Delta$-absorbing.
Clearly, $\widetilde{\AA}(x)$ is compact and the multifunction from $\XX$ to $\AA$ defined by $x\twoheadrightarrow\widetilde{\AA}(x)$ is weakly measurable by Corollary 18.14 in \cite{aliprantis06}
and so, $\widetilde{\mathsf{M}}$ satisfies Condition $(\mathrm{S}_{1})$ in \cite{fra-tom2024}. Moreover, for any $B\in\bfrak{X}$ and $x\in\XX$ the function $Q(B|x,\cdot)$ is continuous on $\widetilde{\AA}(x)$.
This condition differs slightly from $(\mathrm{S}_{2})$ in \cite{fra-tom2024} but it can be shown that Theorem 4.10 in \cite{fra-tom2024} remains valid under this new condition.
Therefore, the set $\widetilde{\bcal{O}}$ is compact for the $ws$-topology.
Remark that the restriction to $\widetilde{\KK}$ of each component of the function
$g$ is a Carath\'eodory function and so,  the subset (possibly empty) $\widetilde{\bcal{O}} \inter \bcal{H}(g,\alpha)$ of $\widetilde{\bcal{O}}$ 
is compact for the $ws$-topology. 
Now suppose that $ \widetilde{\bcal{O}} \inter \bcal{H}(g,\alpha)$ is nonempty then \mbox{$\ext(\widetilde{\bcal{O}} \inter \bcal{H}(g,\alpha))\neq\emptyset$}. Moreover, it is also a face of $\bcal{O}(g,\alpha)$ since 
$\widetilde{\bcal{O}}$ is a face of $\bcal{O}$ and $\bcal{H}(g,\alpha)$ is convex. 
Therefore, $\ext(\widetilde{\bcal{O}} \inter \bcal{H}(g,\alpha))\subset \ext(\bcal{O}(g,\alpha))$ showing the last part of the claim.
\end{proof}

An interesting consequence of Theorem \ref{Castaing-Valadier} is the following result that states that the sufficiency of chattering stationary policies directly yields the sufficiency of deterministic stationary policies for atomless models.
\begin{proposition}
\label{Consequence-chattering}
Assume that $\mathsf{M}$ is uniformly $\Delta$-absorbing and atomless. 
For every $\gamma\in\mathbb{C}$, there exists $\phi\in\mathbb{D}$ satisfying $\mu_{\gamma}(r)=\mu_{\phi}(r)$.
In particular, if $\mathcal{R}(\mathbb{C}_{p})=\mathcal{R}(\mathbf{\Pi})$ then $\mathcal{R}(\mathbb{D})=\mathcal{R}(\mathbf{\Pi})$.
\end{proposition}
\begin{proof}
For $\gamma\in\mathbb{C}_{p}$ for $p\in\NN^{*}$ there exists $\{\gamma_{i}\}_{i\in\NN^*_{p}}\subset \mathbb{D}$
such that $\gamma(\cdot|x)$ is supported on $\{\gamma_{i}(x)\}_{i\in\NN^*_{p}}\subset \AA(x)$.
With $\widetilde{\KK}=\big\{(x,a)\in\XX\times\AA : a\in\{\gamma_{i}(x): i\in\NN^*_{p}\}\}\big\}$, write $\widetilde{\bcal{O}}$ as defined in equation \eqref{Def-Otilde} .
Clearly, $\mu_{\gamma}\in\widetilde{\bcal{O}}$ and so, $\widetilde{\bcal{O}} \inter \bcal{H}(r,\mu_{\pi}(r)) \neq\emptyset$. 
From Lemma \ref{finite-compactness}, we have $\ext\big( \bcal{O}(r,\mu_{\pi}(r)) \big)\neq\emptyset$.
We can now apply Theorem \ref{Castaing-Valadier} to get the existence of a deterministic stationary policy, $\phi\in\mathbb{D}$ satisfying
$\mu_{\phi}(r)=\mu_{\pi}(r)$ showing the result.
\end{proof}

The following theorem represents an important step in the proof of the sufficiency of the sets of policies $\mathbb{C}_{d+1}$ and $\mathbb{D}$.
It shows under appropriate assumptions that $\mathcal{R}(\mathbb{C}_{d+1})$ and $\mathcal{R}(\mathbb{D})$ are convex and dense in $\mathcal{R}(\mathbf{\Pi})$ and all these sets have the same relative interior.
\begin{theorem}
\label{Convexity-Denseness}
Assume that the Markov controlled model $\mathsf{M}$ is uniformly $\Delta$-absorbing.
Then the following properties hold. 
\begin{enumerate}[label=(\alph*)]
\item \label{Convexity-Denseness-a} $\mathcal{R}(\mathbb{C}_{d+1})$ is convex and dense in $\mathcal{R}(\mathbf{\Pi})$ and both have the same relative interior.
\item \label{Convexity-Denseness-b} If in addition $\mathsf{M}$ is atomless, $\mathcal{R}(\mathbb{D})$ is convex and dense in $\mathcal{R}(\mathbf{\Pi})$ and both have the same relative interior.
\end{enumerate}
\end{theorem}
\begin{proof}
Let us show \ref{Convexity-Denseness-a}.
For the first point, consider $\pi_{1}$ and $\pi_{2}$ two chattering stationary policies in $\mathbb{C}_{d+1}$ and $\alpha\in]0,1[$.
Then, there exists two finite family of functions $\{\phi_{1,i}\}_{i\in\NN^*_{d+1}}\subset \mathbb{D}$
and $\{\phi_{2,j}\}_{j\in\NN^*_{d+1}}\subset \mathbb{D}$ such that 
$\pi_{1}(\cdot|x)$ is supported on $\{\phi_{1,i}(x)\}_{i\in\NN^*_{d+1}}\subset \AA(x)$
and $\pi_{2}(\cdot|x)$ is supported on $\{\phi_{2,j}(x)\}_{j\in\NN^*_{d+1}}\subset \AA(x)$.
Let us write $$\widetilde{\KK}=\big\{(x,a)\in\XX\times\AA : a\in\{\phi_{1,i}(x): i\in\NN^*_{d+1}\}\union\{\phi_{2,j}(x): j\in\NN^*_{d+1}\}\big\}$$
and $\widetilde{\bcal{O}}$ as defined in equation \eqref{Def-Otilde}.
Clearly, $\mu_{\pi_{1}}$ and $\mu_{\pi_{2}}$ are in the convex set $\widetilde{\bcal{O}}$ and so, there exists a policy $\pi\in\mathbb{S}$ satisfying $\mu_{\pi}=\alpha\mu_{\pi_{1}}+(1-\alpha)\mu_{\pi_{2}}$ and $\mu_{\pi}(\widetilde{\KK}^c)=0$. Therefore, $\widetilde{\bcal{O}} \inter \bcal{H}(r,\mu_{\pi}(r)) \neq\emptyset$. 
By applying Lemma \ref{finite-compactness}, we have $\ext\big( \bcal{O}(r,\mu_{\pi}(r)) \big)\neq\emptyset$.
We can now apply Theorem \ref{Dubins} to get the existence of a  chattering stationary policy of order $d+1$, $\gamma\in\mathbb{C}_{d+1}$ satisfying
$\mu_{\gamma}(r)=\mu_{\pi}(r)=\alpha\mu_{\pi_{1}}(r)+(1-\alpha)\mu_{\pi_{2}}(r)$ giving the first part of the result.

For the second point, let us consider $\mu_{\pi}(r)$ for $\pi\in\mathbf{\Pi}$. From Theorem \ref{deterministic-Markov-policy}\ref{deterministic-Markov-policy-a}, there is no loss of generality to assume that $\pi$ is given
by a chattering Markov policy denoted by $\{\pi_{n}\}_{n\in\NN}$.
Therefore, for any $\varepsilon >0$ there exists $N\in\NN$ such that
$\ds \sup_{\gamma\in\mathbf{\Pi}} \sum_{k=N+1}^{\infty} \mathbb{E}_{\gamma} \big[| r(X_{k},A_{k})|\big]\leq \varepsilon /2$ since $\mathsf{M}$ is uniformly $\Delta$-absorbing.
Let us define the chattering Markov policy $\rho=\{\rho_{k}\}_{k\in\NN}$ by 
$\rho_{k}=\pi_{k}$ for $k\leq N$ and $\rho_{k}=\pi_{N}$ for $k\geq N+1$. Then we have
$$|\mu_{\pi}(r)-\mu_{\rho}(r)|\leq \varepsilon.$$
For any $n\in\NN$, there exists a finite family of functions $\{\phi_{n,i}\}_{i\in\NN^*_{2d+1}}\subset \mathbb{D}$ such that 
$\pi_{n}(\cdot|x)$ is supported on $\{\phi_{n,i}(x)\}_{i\in\NN^*_{2d+1}}\subset \AA(x)$.
With $\widetilde{\KK}=\Big\{(x,a)\in\XX\times\AA : a\in\union_{n=1}^N \big\{\phi_{n,i}(x): i\in\NN^*_{2d+1} \big\} \Big\}$ and  $\widetilde{\bcal{O}}$ as defined in equation \eqref{Def-Otilde},
we have $\widetilde{\bcal{O}} \inter \bcal{H}(r,\mu_{\rho}(r)) \neq\emptyset$ and so,  $\ext\big( \bcal{O}(r,\mu_{\rho}(r)) \big)\neq\emptyset$ by Lemma \ref{finite-compactness}.
Applying Theorem \ref{Dubins}, there exists a chattering stationary policy $\varphi\in\mathbb{C}_{d+1}$ satisfying $\mu_{\rho}(r)=\mu_{\varphi}(r)$ and so,
$|\mu_{\pi}(r)-\mu_{\varphi}(r)|\leq \varepsilon$
implying $\mathcal{R}(\mathbb{C}_{d+1})$ is dense in $\mathcal{R}(\mathbf{\Pi})$.

Finally, the previous result implies that $\widebar{\mathcal{R}(\mathbf{\Pi})}$ and $\widebar{\mathcal{R}(\mathbb{C}_{d+1})}$ have the same relative interior. Since $\mathcal{R}(\mathbb{C}_{d+1}))$ is convex,
it follows by Theorem 6.3 in \cite{rockafellar70}  that $\mathcal{R}(\mathbf{\Pi})$ and $\mathcal{R}(\mathbb{C}_{d+1})$ also have the same relative interior
showing the first item.

\bigskip

To prove item \ref{Convexity-Denseness-b}, we will proceed in the same way as for item \ref{Convexity-Denseness-a}. We will therefore present a less detailed proof than for \ref{Convexity-Denseness-a}.
Nevertheless, we feel it is important to provide these elements so as not to leave a key result unproven.

For the first point, consider $\phi_{1}$ and $\phi_{2}$ two deterministic stationary policies in $\mathbb{D}$ and $\alpha\in]0,1[$.
Write $\widetilde{\KK}=\big\{(x,a)\in\XX\times\AA : a\in\{\phi_{1}(x),\phi_{2}(x)\}\big\}$ and consider $\widetilde{\bcal{O}}$ as defined in equation \eqref{Def-Otilde}.
There exists a policy $\pi\in\mathbb{S}$ satisfying $\mu_{\pi}=\alpha\mu_{\phi_{1}}+(1-\alpha)\mu_{\phi_{2}}$ and
$\mu_{\pi}(\widetilde{\KK}^c)=0$.
Consequently, $\widetilde{\bcal{O}} \inter \bcal{H}(r,\mu_{\pi}(r)) \neq\emptyset$. By applying Lemma \ref{finite-compactness}, we have $\ext\big( \bcal{O}(r,\mu_{\pi}(r)) \big)\neq\emptyset$.
We can now apply Theorem \ref{Castaing-Valadier} to get the existence of a 
deterministic stationary policy $\phi\in\mathbb{D}$ satisfying $\mu_{\phi}(r)=\mu_{\pi}(r)=\alpha\mu_{\phi_{1}}(r)+(1-\alpha)\mu_{\phi_{2}}(r)$ giving the first part of the result.

Regarding the second claim, let us consider $\mu_{\pi}(r)$ for $\pi\in\mathbf{\Pi}$. From Theorem \ref{deterministic-Markov-policy}\ref{deterministic-Markov-policy-b}, there is no loss of generality to assume that $\pi$ is given by a deterministic Markov policy denoted by $\{\phi_{n}\}_{n\in\NN}$. Since $\mathsf{M}$ is uniformly $\Delta$-absorbing, for any $\varepsilon >0$ there exists $N\in\NN$ and a policy $\Phi^N=\{\Phi^N_{k}\}_{k\in\NN}$ given by 
$\Phi^N_{k}=\phi_{k}$ for $k\leq N$ and $\Phi^N_{k}=\phi_{N}$ for $k\geq N+1$ satisfying $|\mu_{\pi}(r)-\mu_{\Phi^N}(r)|\leq \varepsilon$.
With $\widetilde{\KK}=\big\{(x,a)\in\XX\times\AA : a\in\{\phi_{1}(x),\ldots,\phi_{N}(x)\}\big\}$ and  $\widetilde{\bcal{O}}$ as defined in equation \eqref{Def-Otilde}, we have
$\widetilde{\bcal{O}} \inter \bcal{H}(r,\mu_{\Phi}(r)) \neq\emptyset$ and so,  $\ext\big( \bcal{O}(r,\mu_{\Phi}(r)) \big)\neq\emptyset$ by Lemma \ref{finite-compactness}.
Applying Theorem \ref{Castaing-Valadier}, there exists a  deterministic stationary policy $\varphi\in\mathbb{D}$ satisfying $\mu_{\Phi}(r)=\mu_{\varphi}(r)$ implying
$|\mu_{\pi}(r)-\mu_{\varphi}(r)|\leq \varepsilon$. This shows that $\mathcal{R}(\mathbb{D})$ is dense in $\mathcal{R}(\mathbf{\Pi})$. The last part of the claim follows the same arguments as for item \ref{Convexity-Denseness-a} 
and the proof is complete.
\end{proof}

We are now going to generalize a striking result from \cite{piunovskiy19}. We will establish that Theorem 6.3 in \cite{piunovskiy19} is still valid without assuming that the underlying model is atomless and for a general measurable state space.
\begin{theorem}
\label{Alexey-Theo}
Assume that the model $\mathsf{M}$ is uniformly $\Delta$-absorbing.
For $\gamma\in\mathbf{\Pi}$, there exists a uniformly $\Delta$-absorbing model $\mathsf{M}^{*}=(\mathbf{X},\mathbf{A},\{\mathbf{A}^{*}(x)\}_{x\in \mathbf{X}},Q,\eta,r)$ satisfying the measurability conditions
with $\AA^{*}(x)$ a finite or countable subset of $\AA(x)$ and letting $\bcal{O}^{*}$ be the set of occupation measures of $\mathsf{M}^{*}$ we have
\begin{enumerate}[label=\arabic*)]
\item  $\mu_{\gamma}(r)=\mu^{*}(r)$ for some $\mu^{*}\in\bcal{O}^{*}$ and $\mu\ll \mu^{*}$ for any $\mu\in \bcal{O}^{*}$,
\item If $\ds \sup_{\mu\in\bcal{O}^{*}} G(\mu(r)) = G(\mu^{*}(r))$ for a linear functional $G$ on $\RR^{d}$, then
$G(\mu(r)) = G(\mu^{*}(r))$ for any $\mu\in\bcal{O}^{*}$.
\end{enumerate}
\end{theorem}
\begin{proof}
The first item was proved in Lemma 6.2 of  \cite{piunovskiy19} for an atomless model defined on a Borel state space.
We will start by showing that the assumption of an atomless model can be dropped.
To do this we will use the key result proved in Theorem \ref{deterministic-Markov-policy}\ref{deterministic-Markov-policy-a} namely that 
the set of chattering Markov policies is a sufficient family of policies for the constrained problem.
For $\gamma\in\mathbf{\Pi}$, there exists a chattering Markov policy $\rho=\{\rho_{i}\}_{i\in\NN}$ satisfying
$\mu_{\gamma}(r)=\mu_{\rho}(r)$ from Theorem \ref{deterministic-Markov-policy}\ref{deterministic-Markov-policy-a}.
Observe that for any $n\in\NN$, there exists a finite family of functions $\{\phi_{n,i}\}_{i\in\NN^*_{2d+1}}\subset \mathbb{D}$ such that 
$\gamma_{n}(\cdot|x)$ is supported on $\{\phi_{n,i}(x)\}_{i\in\NN^*_{2d+1}}\subset \AA(x)$.
For $x\in\XX$, let us write 
$$\AA_{\Phi}(x)=\union_{n\in\NN} \big\{\phi_{n,i}(x) : i\in\NN^*_{2d+1} \big\}.$$
By Theorem 4.45 in \cite{aliprantis06}, $\KK_{\Phi}=\{(x,a)\in\XX\times\AA : a\in\AA_{\Phi}(x)\}$ is $\bfrak{X}\otimes\bfrak{B}(\AA)$-measurable.
Then, according to the disintegration Lemma, there exists $\sigma\in\mathbb{S}$ such that $\mu_{\rho}=\mu^{\XX}_{\rho}\otimes\sigma$ and \mbox{$\sigma(\AA_{\Phi}(x) | x) =1$} for any $x\in\XX$.
Proceeding as for the proof of Proposition 3.3(ii) in \cite{fra-tom2024}, it follows that $\mu_{\rho}=\mu_{\sigma}$.
Clearly, 
the set $\KK_{\sigma}=\{(x,a)\in\XX\times\AA : a\in\AA_{\sigma}(x)\}$ is $\bfrak{X}\otimes\bfrak{B}(\AA)$-measurable
where $\AA_{\sigma}(x)=\{a\in\AA_{\Phi}(x) : \sigma(\{a\}|x) >0\}$.
However, we cannot conclude that the graph of $\KK_{\sigma}$ has a measure selector. This is a delicate point which requires an additional step with respect to the proof of Lemma 6.2 of  \cite{piunovskiy19}
 which we will now describe.
From Aumann's selection Theorem (see Theorem III.22 in \cite{castaing77}), there exists a $\bfrak{X}_{\sigma}$-measurable function $\phi$ from $\XX$ to $\AA$ whose graph is a subset of $\KK_{\sigma}$
where $\bfrak{X}_{\sigma}$ denotes the $\mu_{\sigma}^{\XX}$-completion of $\bfrak{X}$. From Lemma 1.2 in \cite{crauel02}, we have $\{x\in \XX : \phi(x) \neq \widetilde{\phi}(x)\}\subset \mathcal{N}\in\bfrak{X}$ with $\mu^{\XX}_{\sigma}(\mathcal{N})=0$
for some $\bfrak{X}$-measurable function $\tilde{\phi}$ from $\XX$ to $\AA$. Let us write $\AA^{*}(x)=\AA_{\sigma}(x)$ for $x\in\mathcal{N}^c$ and $\AA^{*}(x)=\{\theta(x)\}$ for $x\in\mathcal{N}$ with
$\theta$ as introduced  in Assumption \ref{selector-theta}. Then, we have  $\KK^{*}=\big\{(x,a)\in\XX\times\AA : a\in \AA^{*}(x) \big\}\in\bfrak{\XX}\otimes\bfrak{B}(\AA)$.
Moreover, the function $\phi^*$ defined on $\XX$ by $\phi^*(x)=\tilde{\phi}(x)$ for $x\in\mathcal{N}^c$ and $\phi^*(x)=\theta(x)$ for $x\in\mathcal{N}$ is $\bfrak{X}$-measurable and its graph is a subset of $\KK^{*}$.
Now, observe that the stochastic kernel $\sigma^{*}$ on $\AA$ given $\XX$ defined by $\sigma^*(\cdot |x)= \sigma(\cdot |x)$ for $x\in\mathcal{N}^c$ and $\sigma^*( \{\theta(x)\} |x)=1$ for $x\in\mathcal{N}$ satisfies $\sigma^{*}(\AA^{*}(x)|x)=1$ for any $x\in\XX$.
Since $\mu^{\XX}_{\sigma}(\mathcal{N})=0$, we have $\mu_{\sigma}=\sigma\otimes\mu^{\XX}_{\sigma}=\sigma^*\otimes\mu^{\XX}_{\sigma}$ and so,
$\mu_{\sigma}=\mu_{\sigma^*}$ by using  Proposition 3.3(ii) in \cite{fra-tom2024} and consequently, $\mu_{\gamma}(r)=\mu_{\sigma^*}(r)$. 
We also have $\mu\ll\mu_{\sigma^*}$ for any $\mu\in\bcal{O}^{*}$ because $\sigma^*(\{a\}|x)>0$ for any $(x,a)\in\KK^{*}$.
Therefore, the model $\mathsf{M}^{*}$ defined by $(\mathbf{X},\mathbf{A},\{\mathbf{A}^{*}(x)\}_{x\in \mathbf{X}},Q,\eta,r)$ satisfies the first item with $\mu^*=\mu_{\sigma^{*}}$.

\bigskip

The second item follows the proof of Theorem 6.3 in \cite{piunovskiy19} by using the fact that the model $\mathsf{M}^{*}$ is uniformly $\Delta$-absorbing and combining  the projection Theorem
(see Theorem 18.25 in \cite{aliprantis06}) and the Measurable Selection Theorem (see Corollary 18.25 in \cite{aliprantis06}) to deal with a general measurable state space $(\XX,\bfrak{X})$.

Below we provide some additional arguments to supplement those of Theorem 6.3 in \cite{piunovskiy19} to be self-contained.
The underlying model we consider now is $\mathsf{M}^{*}$.
Let $h_{\sigma^{*}}(x)=\sum_{k\in\NN} Q_{\sigma^{*}}^{k}\mathbb{I}_{\Delta^{c}}(\XX | x)$ be the $\bfrak{X}$-measurable function on $\XX$ with values in $\widebar{\RR}_{+}$. Let
\mbox{$\XX_{\sigma^{*}}=\{x\in\XX : h_{\sigma^{*}}(x)<+\infty\}$} and $\KK_{\sigma^{*}}=\KK^{*}\inter[\XX_{\sigma^{*}}\times\AA]$. Since $\mathsf{M}$ is $\Delta$-absorbing, then $\eta(\XX_{\sigma^{*}}^{c})=0$. Since
$Q_{\sigma^{*}}^{k}h_{\sigma^{*}}\leq h_{\sigma^{*}}$ for any $k\in\NN$, this yields $\mu_{\sigma^{*}}^{\XX}(\XX_{\sigma^{*}}^{c})=0$.
Write $g(x,a)=G(r(x,a))$ for any $(x,a)\in\XX\times\AA$ where the linear functional $G$ satisfies the hypothesis of item 2).
Consequently, the function 
$\ds \bar{g}_{\sigma^{*}}(x)=\sum_{k\in\NN} Q_{\sigma^{*}}^{k}\mathbb{I}_{\Delta^{c}}g_{\sigma^{*}}(x)$
 is well defined, finite and 
$\bfrak{X}$-measurable on $\XX_{\sigma^{*}}$. 
Consider the $\bfrak{X}\otimes\bfrak{B}(\AA)$-measurable functions $\ds \int_{\XX_{\sigma^{*}}} \bar{g}_{\sigma^{*}}^{\pm}(y) Q(dy|x,a)$ defined on $\XX\times\AA$. They are finite on
$\KK_{\sigma^{*}}$ since 
$$ \sigma^{*}(\{a\} | x)\int_{\XX_{\sigma^{*}}} \bar{g}_{\sigma^{*}}^{\pm}(y) Q(dy|x,a) \leq \int_{\XX_{\sigma^{*}}} \bar{g}_{\sigma^{*}}^{\pm}(y) Q_{\sigma^{*}}(dy|x)
\leq  \sup_{y\in\XX} |g(y)| \: h_{\sigma^{*}}(x)$$
and $\sigma^{*}(\{a\} | x)>0$ for $(x,a)\in \KK_{\sigma^{*}}$.
Therefore, the function 
$$T_{\sigma^{*}}g(x,a)= g(x,a)+ \int_{\XX_{\sigma^{*}}} \bar{g}_{\sigma^{*}}(y) Q(dy|x,a)$$ is well defined, finite and $\bfrak{X}\otimes\bfrak{B}(\AA)$-measurable on $\KK_{\sigma^{*}}$.
The set 
$$G^{>}=\{(x,a)\in\KK_{\sigma^{*}} : T_{\sigma^{*}}g(x,a) > \bar{g}_{\sigma^{*}}(x) \}$$
is $\bfrak{X}\otimes\bfrak{B}(\AA)$-measurable. From the Projection Theorem \cite[Theorem 18.25]{aliprantis06}, the projection of $G^{>}$ on $\XX$ denoted by $G^{>}_{\XX}$ is a
$\bfrak{X}_{*}$-measurable set where $\bfrak{X}_{*}$ denotes the $\mu_{\sigma^{*}}^{\XX}$-completion of $\bfrak{X}$.
Applying Corollary 18.27 in \cite{aliprantis06}, there exist a \mbox{$\bfrak{X}$-measurable} subset $\XX^{>}$ of $G^{>}_{\XX}$
and a $\bfrak{X}$-measurable set $\bcal{N}^{>}$ such that $G^{>}_{\XX}\subset \XX^{>}\union\bcal{N}^{>}$ and $\mu_{\sigma^{*}}^{\XX}(\bcal{N}^{>})=0$
and a measurable mapping $\varphi^{*}$ from
$\XX^{>}$ to $\AA$ satisfying $\varphi^{*}\in G^{>}(x)$ for any $x\in\XX^{>}$.
Consider an arbitrary $N\in\NN$ and define the Markov policies $\sigma=\{\sigma_n\}_{n\ge0}$ by \mbox{$\sigma_{k}(\cdot |x)=\sigma^{*}_{k}(\cdot |x)$} for $k\neq N$ and 
$\sigma_{N}(\cdot |x) =\sigma^{*}(\cdot | x) \mathbf{I}_{\XX\setminus\XX^{>}}(x)+\delta_{\varphi^{*}(x)}(\cdot) \mathbf{I}_{\XX^{>}}(x)$
for $x\in\XX$.
Now observe that 
\begin{align*}
\int_{\XX^{>}}  Q\bar{g}_{\sigma^{*}}(x,\varphi^{*}(x)) \eta Q^{N}_{\sigma^{*}}(dx) & = \int_{\XX^{>}}  \int_{\XX} \bar{g}_{\sigma^{*}} (y) Q(dy | x,\varphi^{*}(x)) \eta Q^{N}_{\sigma^{*}}(dx) 
\nonumber \\
& = \int_{\XX^{>}}  \int_{\XX_{\sigma^{*}}} \bar{g}_{\sigma^{*}} (y) Q(dy | x,\varphi^{*}(x)) \eta Q^{N}_{\sigma^{*}}(dx) 
\end{align*}
since $\ds \int_{\XX^{>}}  Q(\XX_{\sigma^{*}}^{c} | x,\varphi^{*}(x)) \eta Q^{N}_{\sigma^{*}}(dx) \leq \mu_{\sigma}^{\XX} (\XX^{>})$ and $\mu_{\sigma}^{\XX} \ll \mu_{\sigma^*}^{\XX}=\mu_{\XX}^{*}$
as established in the first item.
Therefore, 
\begin{align*} 
\mu_{\sigma} (g)-\mu_{\sigma^{*}} (g) 
& = \int_{\XX^{>}}  \big[g(x,\varphi^{*}(x)) +Q\bar{g}_{\sigma^{*}}(x,\varphi^{*}(x)) - \bar{g}_{\sigma^{*}}(x)\big] \eta Q^{N}_{\sigma^{*}}(dx) \nonumber\\
& = \int_{\XX^{>}}  \big[T_{\sigma^{*}}g(x,\varphi^{*}(x)) - \bar{g}_{\sigma^{*}}(x)\big] \eta Q^{N}_{\sigma^{*}}(dx).
\end{align*}
Since $\mu_{\sigma} (g)\leq \mu_{\sigma^{*}} (g)$ by hypothesis and  $T_{\sigma^{*}}g(x,\varphi^{*}(x)) - \bar{g}_{\sigma^{*}}(x)>0$ for $x$ in $\XX^{>}$ we get that 
$\eta Q^{N}_{\sigma^{*}}(\XX^{>})=0$ and so, $\mu_{\sigma^{*}}^{\XX}(\XX^{>})=0$.

Proceeding as above, there exist a $\bfrak{X}$-measurable set $\bcal{N}^{<}$ and
a \mbox{$\bfrak{X}$-measurable} subset $\XX^{<}$ of $G^{<}_{\XX}$ defined as the projection on $\XX$ of the set $G^{<}=\{(x,a)\in\KK_{\sigma^{*}} : T_{\sigma^{*}}g(x,a) < \bar{g}_{\sigma^{*}}(x) \}$
such that $\mu_{\sigma^{*}}^{\XX}(\bcal{N}^{<})=0$ and $G^{<}_{\XX}\subset \XX^{<}\union\bcal{N}^{<}$.
For an arbitrary $N\in\NN$,
\begin{align*} 
\mu_{\sigma^*} (g)-\mu_{\sigma^{*}} (g) 
& = \int_{\XX^{<}\setminus\XX^{>}}  \Big[\sum_{a\in\AA^{*}(x)} \sigma^*(\{a\}|x)T_{\sigma^{*}}g(x,a) - \bar{g}_{\sigma^{*}}(x)\Big] \eta Q^{N}_{\sigma^{*}}(dx).
\end{align*}
The expression in the bracket of the previous equation is strictly negative and so $\eta Q^{N}_{\sigma^{*}}(\XX^{<}\setminus\XX^{>})=0$ yielding
$\mu_{\sigma^{*}}^{\XX}(\XX^{<}\setminus\XX^{>})=0$.

Finally, we obtain $\mu_{\sigma^{*}}^{\XX}(\XX^{=})=\mu_{\sigma^{*}}^{\XX}(\XX)$ for $\XX^{=}=\XX_{\sigma^{*}}\setminus[\XX^{>}\union\XX^{<}]$.
The proof is easily completed by following the arguments presented on page 177 in \cite{piunovskiy19}.
\end{proof}

The next theorem states that for any policy $\pi\in\mathbf{\Pi}$, there exist a subset 
$\mathbf{\Pi}^{*}$ of $\mathbf{\Pi}$ such that the model $\mathsf{M}^{*}$ associated with $\mathbf{\Pi}^{*}$ inherits the properties of $\mathsf{M}$ and for which also
$\mu_{\pi}(r)$ is in the relative interior of $r(\mathbf{\Pi}^{*})$.
\begin{theorem}
\label{Interior-point}
Assume that the model $\mathsf{M}$ is uniformly $\Delta$-absorbing.
For any $\pi\in\mathbf{\Pi}$, there exists a uniformly $\Delta$-absorbing model $\mathsf{M}^{*}=(\mathbf{X},\mathbf{A},\{\mathbf{A}^*(x)\}_{x\in \mathbf{X}},Q,\eta,r)$ satisfying the measurability conditions with ${\AA}^*(x)\subset\AA(x)$
and for which $\mu_{\pi}(r)$ belongs to the relative interior of $\mathcal{R}(\mathbf{\Pi}^{*})$ where $\mathbf{\Pi}^{*}\subset\mathbf{\Pi}$ denotes set of all control policies associated with the model $\mathsf{M}^{*}$.
\end{theorem}
\begin{proof}
According to Theorem \ref{Alexey-Theo}, for any $\pi\in\mathbf{\Pi}$, there
exists a uniformly $\Delta$-absorbing model $\mathsf{M}^{*}=(\mathbf{X},\mathbf{A},\{\mathbf{A}^{*}(x)\}_{x\in \mathbf{X}},Q,\eta,r)$ satisfying the measurability conditions with $\AA^{*}(x)\subset\AA(x)$ and such that
$\mu_{\pi}(r)\in \mathcal{R}(\mathbf{\Pi}^{*})$ where $\mathbf{\Pi}^{*}$ denotes set of all control policies associated with the model $\mathsf{M}^{*}$.
Clearly, $\mathbf{\Pi}^{*}\subset\mathbf{\Pi}$.
Suppose that $\mu_{\pi}(r)$ belongs to the relative boundary of $\mathcal{R}(\mathbf{\Pi}^{*})$. Then, there is a linear functional $G$ on $\RR^d$ not constant on $\mathcal{R}(\mathbf{\Pi}^{*})$ such that
$G$ achieves its maximum at $\mu_{\pi}(r)$ by Corollary 11.6.2 in \cite{rockafellar70}. 
However, we have by Theorem \ref{Alexey-Theo} that $G(\mu(r))= G(\mu_{\pi}(r))$ for any $\mu\in \bcal{O}^{*}$.
This leads a contradiction and so $\mu_{\pi}(r)$ belongs necessarily to the relative interior of $\mathcal{R}(\mathbf{\Pi}^{*})$.
\end{proof}

Now, we obtain the main results of this section.
\begin{theorem}
\label{Main-theorem-aa}
Assume that the model $\mathsf{M}$ is uniformly $\Delta$-absorbing.
Then, $\mathbb{C}_{d+1}$ is a sufficient family of policies, that is, $\mathcal{R}(\mathbb{C}_{d+1})=\mathcal{R}(\mathbf{\Pi})$.
\end{theorem}
\begin{proof} 
For an arbitrary policy $\pi\in\mathbf{\Pi}$,  there exists $\mathsf{M}^{*}=(\mathbf{X},\mathbf{A},\{\mathbf{A}^*(x)\}_{x\in \mathbf{X}},Q,\eta,r)$ 
a uniformly $\Delta$-absorbing model satisfying the measurability conditions with
${\AA}^*(x)\subset\AA(x)$ for which $\mu_{\pi}(r)$ belongs to the relative interior of $\mathcal{R}(\mathbf{\Pi}^{*})$. Let us write $\mathbb{C}_{d+1}^*$ for the set of chattering stationary policies for the model $\mathsf{M}^{*}$.
We can apply Theorem \ref{Convexity-Denseness}\ref{Convexity-Denseness-a} to get that $\mathcal{R}(\mathbf{\Pi}^{*})$ and $\mathcal{R}(\mathbb{C}_{d+1}^*)$ have the same relative interior.
Therefore, there exists a chattering stationary policy  $\gamma\in\mathbb{C}_{d+1}^*$ such that $\mu_{\pi}(r)=\mu_{\gamma}(r)$ but $\mathbb{C}_{d+1}^*\subset\mathbb{C}_{d+1}$ showing the result.
\end{proof}
We provide two different proofs for the second main result which is an extension of the Feinberg-Piunovskiy Theorem in the context of a general state space
\begin{theorem}[The Feinberg-Piunovskiy Theorem]
\label{Main-theorem-bb}
Assume that the model $\mathsf{M}$ is uniformly $\Delta$-absorbing and atomless then $\mathbb{D}$ is a sufficient family of policies, that is, $\mathcal{R}(\mathbb{D})=\mathcal{R}(\mathbf{\Pi})$.
\end{theorem}
\begin{proof}
The first proof is obtained by combining Proposition \ref{Consequence-chattering} and Theorem \ref{Main-theorem-aa}.
\nl
The second proof consists in proceeding exactly as in the proof of Theorem \ref{Main-theorem-aa} by applying Theorem \ref{Convexity-Denseness}\ref{Convexity-Denseness-b} and by noting that if $\mathsf{M}$ is atomless so is $\mathsf{M}^*$ and also $\mathbb{D}^*\subset\mathbb{D}$ if $\mathbb{D}^*$ denotes the set of deterministic stationary policies for the model $\mathsf{M}^{*}$.
\end{proof}



\end{document}